\documentclass[12pt]{amsart}
\usepackage{amssymb}
\usepackage{hyperref}
\theoremstyle{plain}
\newtheorem{theorem}{Theorem}[section]
\newtheorem{proposition}[theorem]{Proposition}
\newtheorem{corollary}[theorem]{Corollary}
\newtheorem{lemma}[theorem]{Lemma}

\theoremstyle{definition}

\newtheorem{example}[theorem]{Example}
\newtheorem{remark}[theorem]{Remark}

\newtheorem{conjecture}[theorem]{Conjecture}

\newtheorem{notation}[theorem]{Notation}

\theoremstyle{remark}

\numberwithin{equation}{section}

\newcommand{\N}{\mathbb N}
\newcommand{\Z}{\mathbb Z}

\newcommand{\R}{\mathbb R}
\newcommand{\C}{\mathbb C}

\newcommand{\GL}{\operatorname{GL}}

\newcommand{\Ot}{\operatorname{O}}
\newcommand{\SO}{\operatorname{SO}}
\newcommand{\SU}{\operatorname{SU}}

\newcommand{\sll}{\mathfrak{sl}}

\newcommand{\su}{\mathfrak{su}}

\newcommand{\op}{\operatorname}

\newcommand{\spec}{\operatorname{Spec}}

\newcommand{\diag}{\operatorname{diag}}

\newcommand{\extremo}{\alpha}
\newcommand{\abc}{(a,b,c)}
\newcommand{\xyz}{(x,y,z)}

\title[The smallest Laplace eigenvalue]{The smallest Laplace eigenvalue of homogeneous 3-spheres}
\author{Emilio~A.~Lauret}
\address{CIEM--FaMAF (CONICET), Universidad Nacional de C\'ordoba, Medina Allende s/n, Ciudad Universitaria, 5000 C\'ordoba, Argentina.}
\email{elauret@famaf.unc.edu.ar}
\thanks{This work was carried out while the author was an Alexander von Humboldt Foundation Postdoctoral Fellow at Humboldt-Universit\"at zu Berlin between April 2017 and May 2018. 
The support and hospitality are gratefully acknowledged.
Supported also by grants from CONICET and FONCyT}
\date{July 24, 2018}

\begin{document}

\begin{abstract}
We establish an explicit expression for the smallest non-zero eigenvalue of the Laplace--Beltrami operator on every homogeneous metric on the 3-sphere, or equivalently, on SU(2) endowed with  left-invariant metric. 
For the subfamily of 3-dimensional Berger spheres, we obtain a full description of their spectra. 

We also give several consequences of the mentioned expression. 
One of them improves known estimates for the smallest non-zero eigenvalue in terms of the diameter for homogeneous 3-spheres. 
Another application shows that the spectrum of the Laplace--Beltrami operator distinguishes up to isometry any left-invariant metric on SU(2). 
It is also proved the non-existence of constant scalar curvature metrics conformal and arbitrarily close to any non-round homogeneous metric on the 3-sphere. 

All of the above results are extended to left-invariant metrics on SO(3), that is, homogeneous metrics on the 3-dimensional real projective space. 
\end{abstract}

\maketitle

\tableofcontents

\section{Introduction}\label{sec:intro}

Let $(M,g)$ be a compact connected Riemannian manifold without boundary.
Let us denote by $\spec(M,g)$ the spectrum of the Laplace--Beltrami operator $\Delta_{g}$ on $(M,g)$. 
It is well known that $\spec(M,g)$ consists of a multiset of eigenvalues
\begin{equation}\label{eq1:spectrum(M,g)}
0=\lambda_0(M,g)<\lambda_1(M,g)\leq \lambda_2(M,g)\leq \lambda_3(M,g)\leq \cdots,
\end{equation}
satisfying that $\lambda_j(M,g)\to\infty$ when $j\to \infty$, so that each of them has finite multiplicity.

An explicit description of the spectrum of $\Delta_g$ is known only for very particular choices of $(M,g)$.
Most of these choices are covered by homogeneous Riemannian manifolds. 
Even expressions for the smallest non-zero eigenvalue are not avaiable in general, though there exist several estimates for it under different geometric conditions (cf.\ in \cite[\S9.10]{Berger-panoramic} or \cite{LingLu10}).

The primary goal of this article is to give an explicit expression of $\lambda_1(S^3,g)$ for any homogeneous Riemannian metric $g$ on the $3$-dimensional sphere $S^3$. 
%We next introduce some notation to state the mentioned expression. 

The group $\SU(2)$ is diffeomorphic to $S^3$. 
Homogeneous metrics on $S^3$ are invariant by $\SU(2)$.
Consequently, they are in correspondence with left-invariant metrics on $\SU(2)$. 
Moreover, round metrics on $S^3$ correspond to bi-invariant metrics on $\SU(2)$. 
The space of left-invariant metrics on $\SU(2)$ is parametrized by three positive real numbers $a$, $b$ and $c$. 
Indeed, if $\{X_1,X_2,X_3\}$ is an orthonormal basis of $\su(2)$ with respect to a fixed bi-invariant metric, then any left-invariant metric is isometric to some metric $g_{\abc}$ induced by the inner product on $\su(2)$ with orthonormal basis $\{aX_1,bX_2,cX_3\}$. 
In other words, the matrix of $g_{\abc}$ on the tangent space at the identity, which is identified with $\su(2)$, with respect to the above basis, is given by
\begin{equation}\label{eq2:g_matrix}
[g_{\abc}(X_i,X_j)]_{i,j} = 
\begin{pmatrix}
1/a^2 \\ & 1/b^2 \\ && 1/c^2
\end{pmatrix}.
\end{equation}
See Subsection~\ref{subsec:left-invmet} for more details, specially for our particular choice of $\{X_1,X_2,X_3\}$ in \eqref{eq2:X1X2X3}, which gives constant sectional curvature $1/a^2$ to $(\SU(2),g_{(a,a,a)})$ for any $a>0$. 
It turns out that the isometry class of $g_{\abc}$ is invariant by any permutation of $\abc$, so we will usually assume $a\geq b\geq c>0$ without losing generality.

The next result gives an explicit expression of $\lambda_1(S^3,g)$ for every homogeneous metric $g$ on $S^3$ in terms of the parameters $a,b,c$ introduced above.

\begin{theorem}\label{thm1:SUlambda_1}
Under the assumption $a\geq b\geq c>0$, the smallest non-zero eigenvalue of the Laplace--Beltrami operator on $\SU(2)$ endowed with the left-invariant metric $g_{\abc}$ is given by
\begin{equation*}
\lambda_1(\SU(2),g_{\abc})=
\min\{a^{2} + b^{2} + c^{2},\; 4(b^2+c^2)\}.
\end{equation*}
Moreover, the multiplicity of $\lambda_1(\SU(2),g_{\abc})$ in $\spec(\SU(2),g_{\abc})$ is equal to four if $a^{2} + b^{2} + c^{2} <4(b^2+c^2)$, seven if $a^{2} + b^{2} + c^{2} =4(b^2+c^2)$, and three if $a^{2} + b^{2} + c^{2} >4(b^2+c^2)$. 
\end{theorem}

Theorem~\ref{thm3:nexteigenvalues} gives more information on the next eigenvalues in $\spec(\SU(2),g_{\abc})$. 
The proofs of Theorems~\ref{thm1:SUlambda_1} and \ref{thm3:nexteigenvalues} use the algebraic machinery of Lie theory. 
Using the same approach, Urakawa had shown that $\lambda_1(\SU(2),g_{\abc})\leq a^{2} + b^{2} + c^{2}$ (see \cite[Thm.~5]{Urakawa79}). 
Furthermore, given any non-commutative compact Lie group $G$, he constructed a continuous family $g_t$ ($t>0$) of left-invariant metrics of constant volume such that $\lim_{t\to0} \lambda_1(G,g_{t})=0$ and $\lim_{t\to\infty } \lambda_1(G,g_{t})=\infty$ (see \cite[Thm.~4]{Urakawa79}). 
In the case where $G=\SU(2)$, this deformation can be taken to occur within the family of Berger metrics on $\SU(2)$; i.e.\ the family of left-invariant metrics $g_{(a,b,c)}$ with at least two of $a,b,c$ equal.

Inspired by \cite{Urakawa79}, B\'erard-Bergery and Bourguignon~\cite{Berard-BergeryBourguignon82} gave a general method to compute the spectrum of a total space of a Riemannian submersion with totally geodesic fibers. 
Indeed, Bettiol and Piccione~\cite{BettiolPiccione13a} used this method to give explicit expressions for the lowest positive eigenvalue of several homogeneous spheres, including three-dimensional Berger spheres. 
Of course, their expression in dimension three coincides with the one in \cite{Urakawa79} and Theorem~\ref{thm1:SUlambda_1}; however, it is important to note that this method cannot be applied to an arbitrary left-invariant metric on $\SU(2)$.

We will also consider the Lie group $\SO(3)$, which is isomorphic to $\SU(2)/\{\pm I_2\}$ and diffeomorphic to the $3$-dimensional real projective space $P^3(\R)$.
The space of left-invariant metrics on $\SO(3)$ is parametrized in the same way as for $\SU(2)$.

\begin{theorem}\label{thm1:SOlambda_1}
Under the assumption $a\geq b\geq c>0$, we have that
\begin{equation*}
\lambda_1(\SO(3),g_{\abc})=
4(b^2+c^2),
\end{equation*}
with multiplicity three if $a>b$, six if $a=b>c$, and nine if $a=b=c$. 
\end{theorem}

Furthermore, Proposition~\ref{prop3:specBerger} gives a complete description of the spectrum of any Berger sphere $(\SU(2),g_{(a,b,b)})$ and any $(\SO(3),g_{(a,b,b)})$ for all positive numbers $a,b$.

We now derive several applications of the above expressions to different problems in global analysis. 
The first one deals with estimates for some homogeneous Riemannian manifolds of the smallest non-zero eigenvalue of the Laplace--Beltrami operator in terms of the diameter. 
The second application will show that any metric on $\SU(2)$ or $\SO(3)$ is uniquely determined by its spectrum among all left-invariant metrics. 
The last applications concern the problem of multiple solutions of the Yamabe problem in conformal classes of homogeneous metrics on $S^3$.

\subsection{Applications to estimates} \label{subsec1:estimates}
It is very desirable to have good estimates for $\lambda_1(M,g)$ in terms of geometric quantities. 
We restrict our attention to estimates of the scale-invariant term $\lambda_1(M,g) \op{diam}(M,g)^{2}$, where $\op{diam}(M,g)$ denotes the diameter of $(M,g)$.
Significant contributions to this problem were made by Lichnerowicz~\cite{Lichnerowicz}, Obata~\cite{Obata62}, Li and Yau~\cite{LiYau80}, and Zhong and Yang~\cite{ZhongYang84}.

A common assumption in the above contributions is a lower bound for the Ricci curvature of $(M,g)$. 
Peter Li~\cite{Li80} proved\footnote{During the review process, the author found the recent article \cite{JudgeLyons17} by Judge and Lyons, which improves the lower bound in \eqref{eq1:Li-estimate}.} for any homogeneous compact Riemannian manifold $(M,g)$ that
\begin{equation}\label{eq1:Li-estimate}
\lambda_1(M,g) \geq \frac{\pi^2/4}{\op{diam}(M,g)^2}.
\end{equation} 
In other words, the map $g\mapsto \lambda_1(M,g) \op{diam}(M,g)^{2}$ is bounded from below by $\pi^2/4$ when we restric $g$ to the space of homogeneous metrics on a fixed manifold $M$. 
Moreover, this lower bound is universal in the sense that it does not depend on $M$. 
Concerning upper bounds, 
Eldredge, Gordina and Saloff-Coste recently made the following conjecture (see \cite[(1.2)]{EldredgeGordinaSaloff17}).

\begin{conjecture}[Eldredge, Gordina and Saloff-Coste] \label{conj:EGS}
For each compact Lie group $G$, there is a positive real number $C_G$ such that, for any left-invariant metric $g$ on $G$, one has 
\begin{equation}\label{eq1:EGSconj}
\lambda_1(G,g)  \leq \frac{C_G}{\op{diam}(G,g)^2}.
\end{equation}
\end{conjecture}

They proposed in \cite{EldredgeGordinaSaloff17} a detailed method to prove it. 
This method requires showing the existence of a uniform upper bound of the volume doubling constant among the space of left-invariant metrics. 
Moreover, they proved (among many other things) that this condition holds for $G=\SU(2)$, and therefore the above conjecture holds in this case (see \cite[Thm.~8.5]{EldredgeGordinaSaloff17}).

Notice there is no universal lower bound for the Ricci curvature among the metrics in \eqref{eq1:Li-estimate} and \eqref{eq1:EGSconj}, even after rescaling to constant volume or diameter.   
It might be argued that homogeneity successfully replaces the assumption of a lower bound on the Ricci curvatures.

The next theorem, which is the main goal of this subsection, gives explicit positive bounds for $\lambda_1(G,g) \op{diam}(G,g)^{2}$ valid for every left-invariant metric $g$ on $G=\SU(2)$ and $G=\SO(3)$.

\begin{theorem}\label{thm1:conjSU2SO3}
For every left-invariant metric $g$ on $\SU(2)$ or $\SO(3)$, we have that 
\begin{align}\label{eq1:conjSU(2)}
\frac{\pi^2}{\op{diam}(\SU(2),g)^2} &<\lambda_1(\SU(2),g)  \leq \frac{8\pi^2}{\op{diam}(\SU(2),g)^2},
\\
\label{eq1:conjSO(3)}
\frac{\pi^2}{\op{diam}(\SO(3),g)^2} &<\lambda_1(\SO(3),g)  \leq \frac{(9-4\sqrt{2})\pi^2}{\op{diam}(\SO(3),g)^2}.
\end{align}
\end{theorem}

Notice $3.343 <9-4\sqrt{2}<3.344$.
An explicit expression for the diameter of $(\SU(2),g_{\abc})$ is known only for Berger spheres (see Proposition~\ref{prop4:diameterBerger}). 
This expression was a fundamental tool to prove the previous estimates. 
It might be hoped that an extrema of the function
\begin{equation}
g\mapsto\lambda_1(\SU(2),g)\op{diam}(\SU(2),g)^2
\end{equation}
over the space of left-invariant metrics on $\SU(2)$ is attained by a Berger sphere. 
In this case, Corollary~\ref{cor4:BergerSU} would give sharp estimates since, for every Berger sphere, it yields
\begin{equation}
\frac{(1+\sqrt{3}/2)\pi^2}{\op{diam}(\SU(2),g)^2} \leq \lambda_1(\SU(2),g)   \leq \frac{3\pi^2}{\op{diam}(\SU(2),g)^2}.
\end{equation}

Despite the existence of the uniform lower bound in \eqref{eq1:Li-estimate}, Example~\ref{ex4:no-unif-upper-bound} shows that the constant $C_G$ in Conjecture~\ref{conj:EGS} cannot be uniform for every $G$. 
That is, $C_G$ is not bounded from above for every compact Lie group $G$. 
The elementary example uses Proposition~\ref{prop4:estimate-product}, which estimates the term $\lambda_1(G,g) \op{diam}(G,g)^2$ for every Riemannian manifold $(G,g)$ given by a direct product of Riemannian manifolds, where each copy is isometric to a left-invariant metric on $\SU(2)$ or $\SO(3)$. 
It is important to note that the set of metrics $g$ covered above is strictly included in the space of left-invariant metrics on $G$ (see for instance \cite[\S2]{NikonorovRodionov03} for the space of left-invariant metrics on $\SU(2)\times\SU(2)$).

%$<   >$

\subsection{Application to inverse spectral geometry} \label{subsec1:rigidity}

The next result is the second main consequence of Theorems~\ref{thm1:SUlambda_1} and \ref{thm1:SOlambda_1}, namely, the global spectral rigidity of left-invariant metrics on $\SU(2)$ and on $\SO(3)$. 
This theorem was already proved by Schmidt and Sutton~\cite{SchmidtSutton13}\footnote{At the moment, this article can be found in Schmidt's web page and dates from October 16th, 2014. By a personal communication with C.~Sutton, there will be soon a version of this preprint in arxiv with an extended main result.}. 
They used the analytic machinery of heat invariants.

\begin{theorem}\label{thm1:unicidad}
Let $G$ be either $\SU(2)$ or $\SO(3)$. 
The spectrum of the Laplace--Beltrami operator distinguishes up to isometry any left-invariant metric on $G$ within the space of left-invariant metrics on $G$. 
Equivalently, two isospectral left-invariant metrics on $G$ are isometric. 
\end{theorem}

The proof uses two classical spectral invariants (volume and total scalar curvature) plus the smallest non-zero eigenvalue with its multiplicity given in Theorems~\ref{thm1:SUlambda_1} and \ref{thm1:SOlambda_1}.

%$<   >$

\subsection{Application to multiple solutions of the Yamabe problem} \label{subsec1:Yamabe}
We conclude the article with an application of Theorem~\ref{thm1:SUlambda_1} to a problem related with certain solutions of the Yamabe problem on $S^3$. 
This application slightly extends a previous result by Bettiol and Piccione~\cite{BettiolPiccione13a}.
We prove that (up to scaling) within the conformal class of a homogeneous metric $g$ of non-constant sectional curvature on $S^3$, all other metrics of constant scalar curvature must be sufficiently ``far'' from $g$.

The Yamabe problem, which is already solved, predicts for a given compact manifold $M$ of dimension $\geq3$ the existence of a constant scalar curvature metric in each conformal class of metrics on $M$. 
A constant scalar curvature metric is characterized for being a critical point of the Hilbert--Einstein functional restricted to its conformal class.
The Yamabe problem was solved by showing that this functional attains a minimum in each conformal class. 
See \cite[\S1]{LeeParker87} for details and references on this subject.

Anderson~\cite{Anderson05} proved that, generically, there is a unique constant scalar curvature metric (up to scaling) in each conformal class. 
In contrast, the solutions of the Yamabe Problem within the conformal class of the round metric on $S^m$, $m\geq3$, forms an $(m+1)$-dimensional manifold. 
In particular, modulo scaling, the round metric is not an isolated solution of the Yamabe Problem in its conformal class.
%However, there exist multiple solutions in some cases.
%For instance, in the conformal class of a round sphere of dimension $m$, the space of constant scalar curvature metrics has dimension $m+1$. 

Bettiol and Piccione~\cite{BettiolPiccione13a} studied the problem of existence or non-existence of multiple solutions of the Yamabe problem for the homogeneous metrics on spheres (see also \cite{BettiolPiccione13b} for a more general setting). 
Of course, homogeneous metrics are solutions of the Yamabe problem since they have constant scalar curvature.
For dimension $3$, the case of interest in this article, Bettiol and Piccione proved the following result in the case of any non-round Berger sphere (see \cite[Thm.~5.4]{BettiolPiccione13a}).

\begin{theorem}\label{thm1:Yamabe}
Let $g$ be either a left-invariant metric on $\SU(2)$ with non-constant sectional curvature or any left-invariant metric on $\SO(3)$.
Then, there is a neighborhood of $g$ inside the conformal class $[g]$ containing (up to scaling) no other constant scalar curvature metric. 
\end{theorem}

\medskip

The article is organized as follows. 
Section~\ref{sec:preliminaries} recalls well-known spectral and geometric properties of left-invariant metrics on compact Lie groups. 
The next section includes the main results on the partial description of the spectrum of the Laplace--Beltrami operator on $\SU(2)$ or $\SO(3)$ endowed with a left-invariant metric. 
Sections~\ref{sec:estimates}, \ref{sec:uniquiness} and \ref{sec:Yamabe} considers the applications described in Subsections~\ref{subsec1:estimates}, \ref{subsec1:rigidity} and \ref{subsec1:Yamabe} respectively.

\subsection*{Acknowledgments}

The author wishes to express his gratitude
to Juan Pablo Agnelli and Eren U\c{c}ar for drawing the author's attention to the Gershgorin Circle Theorem (used in  Lemma~\ref{lem3:lowerboundCasimireigenvalues}) and the Muirhead's inequality (used in Proposition~\ref{prop6:SUScalcondition}) respectively,
to Nathaniel Eldredge for several helpful comments on the diameter of a left-invariant metric,
to Guillermo Henry for a stimulating conversation on the Yamabe problem,
and specially to Dorothee Sch\"uth for many fruitful discussions on several points in the article.
%The author also wishes to thank the Alexander von Humboldt Foundation for financial support. 
The author is greatly indebted to the referee for a very careful reading and many helpful suggestions.

\section{Preliminaries}\label{sec:preliminaries}

In this section, we review well-known facts.
We begin by recalling an abstract description of the spectrum of the Laplace--Beltrami operator on a compact connected semisimple Lie group endowed with a left-invariant metric.
We conclude the section with a description of the space of left-invariant metrics on $3$-dimensional non-commutative compact Lie groups (e.g.,\ $\SU(2)$ and $\SO(3)$).
We will also recall some well-known facts of their geometry.

\subsection{Spectra of left-invariant metrics} \label{subsec:specleft-inv}
The material in this subsection is extracted from \cite[\S2]{Lauret-globalrigid}, which is based on \cite[\S2 and \S3]{Urakawa79}.

Let $G$ be a compact connected semisimple Lie group of dimension $m$. 
It is well known that left-invariant metrics on $G$ are in correspondence with inner products on the Lie algebra $\mathfrak g$ of $G$. 
We fix an $\op{Ad}(G)$-invariant inner product $\langle \cdot,\cdot\rangle_I$ on $\mathfrak g$ (e.g.\ a negative multiple of the Killing form).
Let $\{X_1,\dots,X_m\}$ be an orthonormal basis of $\mathfrak g$ with respect to $\langle \cdot,\cdot \rangle_I$. 
For $A=(a_{i,j})\in\GL(m,\R)$, we denote by $\langle \cdot,\cdot\rangle_A$ the inner product on $\mathfrak g$ satisfying that $\{Y_1,\dots,Y_m\}$ is an orthonormal basis, where $Y_j=\sum_{i=1}^{m}a_{i,j}X_i$ for any $j$. 
One can check that $\langle \cdot,\cdot \rangle_{A} = \langle \cdot,\cdot \rangle_{B}$ if and only if $A=BP$ for some $P\in\Ot(m)$. 
Consequently, the space of left-invariant metric on $G$ is in correspondence with $\GL(m,\R)/\Ot(m)$. 
For $A\in\GL(m,\R)$, we denote by $(G,g_A)$ the Riemannian manifold $G$ endowed with the left-invariant metric $g_A$ associated to $A$.

Let $(R,L^2(G))$ be the right regular representation of $G$, that is, for each $g\in G$, $R_g:L^2(G)\to L^2(G)$ is unitary given by $f\mapsto (R_g\cdot f)(x)=f(xg)$ for $x\in G$ and $f\in L^2(G)$. 
The Peter-Weyl Theorem ensures the equivalence 
\begin{equation}\label{eq2:PeterWeyl}
L^2(G)\simeq \bigoplus_{\pi\in \widehat G} V_\pi\otimes V_\pi^*
\end{equation}
as $G$-modules.
Here, $\widehat G$ denotes the unitary dual of $G$, the action of $G$ on $V_\pi\otimes V_\pi^*$ is given by $g\cdot (v\otimes\varphi )= v\otimes (\pi^*(g)\varphi)$ (note that this action was given incorrectly in \cite[\S2]{Lauret-globalrigid}), and the embedding $V_\pi\otimes V_\pi^*\hookrightarrow C^\infty(G)$ is given by $v\otimes \varphi\mapsto f_{v\otimes\varphi}(x) = \varphi(\pi(x)v)$ for  $x\in G$.

Set $C_A=\sum_{j=1}^m Y_j^2 \in U(\mathfrak g_\C)$, where $\{Y_1,\dots,Y_m\}$ is any orthonormal basis of $\mathfrak g$ with respect to $\langle\cdot,\cdot\rangle_A$.
One can check that this does not depend on the basis. 
Here, $U(\mathfrak g_\C)$ stands for the universal enveloping algebra associated to $\mathfrak g_\C:=\mathfrak g\otimes_\R\C$. 
We call $C_A$ the Casimir element associated to $\langle\cdot,\cdot \rangle_A$.
We use $\pi$ to denote also the induced representations of $\mathfrak g$, $\mathfrak g_\C$ and $U(\mathfrak g_\C)$. 

Let $\Delta_A$ be the Laplace--Beltrami operator of $(G,g_A)$. 
One has that (c.f.\ \cite[Lem.~1]{Urakawa79}) 
\begin{equation}\label{eq2:Laplacian}
\Delta_A\cdot f_{v\otimes\varphi} = f_{(\pi(-C_A)v)\otimes\varphi}.
\end{equation}
Let $v$ be an eigenvector with eigenvalue say $\lambda$ of the finite-dimensional linear operator $\pi(-C_A):V_\pi\to V_\pi$. 
Then, $\Delta_A \cdot f_{v\otimes\varphi} = f_{(\lambda v)\otimes \varphi}=\lambda f_{v\otimes\varphi}$, that is, $f_{v\otimes \varphi}$ is an eigenfunction of $\Delta_A$ with eigenvalue $\lambda$ for any $\varphi\in V_\pi^*$. 
Consequently, if $\lambda_{1}^{\pi,A},\dots, \lambda_{d_\pi}^{\pi,A}$ denote the eigenvalues of the finite-dimensional linear operator $\pi(-C_A):V_\pi\to V_\pi$ and $d_\pi=\dim V_\pi$, then
\begin{equation}\label{eq2:spec_A}
\spec (\Delta_A) = \bigcup_{\pi\in\widehat G} \big\{\!\big\{
\underbrace{\lambda_1^{\pi,A},\dots, \lambda_1^{\pi,A}}_{d_\pi\text{-times}},\dots, \underbrace{\lambda_{d_\pi}^{\pi,A},\dots,\lambda_{d_\pi}^{\pi,A}}_{d_\pi\text{-times}} 
\big\}\!\big\}.
\end{equation}
We thus have the following result.

\begin{proposition}\label{prop2:spec_A}
If $\lambda$ is an eigenvalue of $\Delta_A$, then there are $\pi\in\widehat G$ and $1\leq j\leq d_\pi$ such that $\lambda=\lambda_j^{\pi,A}$. 
Moreover, the multiplicity of an eigenvalue $\lambda$ of $\Delta_{A}$ is given by 
\begin{equation}\label{eq2:multeigenvalue}
\sum_{(\pi,j)\in \mathcal E_\lambda(G,A)} d_\pi, 
\qquad 
\left(\text{where } \mathcal E_\lambda(G,A):=\{(\pi,j): \pi\in\widehat G,\; 1\leq j\leq d_\pi,\; \lambda=\lambda_{j}^{\pi,A}\}\right).
\end{equation}
\end{proposition}

\subsection{Left-invariant metrics on 3-dimensional compact Lie groups} \label{subsec:left-invmet}
We now review well-known facts about left-invariant metrics on $\SU(2)$ and $\SO(3)$. 
We set 
\begin{align}\label{eq2:X1X2X3}
X_1&= \begin{pmatrix} i&0 \\ 0&-i \end{pmatrix}, &
X_2&= \begin{pmatrix} 0&1 \\ -1&0 \end{pmatrix}, &
X_3&= \begin{pmatrix} 0&i \\ i&0 \end{pmatrix}.
\end{align}
One can easily check that $\{X_1,X_2,X_3\}$ is an orthonormal basis of $\su(2)$ with respect to the $\op{Ad}(\SU(2))$-invariant inner product $\langle X,Y\rangle_I := -\frac18B(X,Y)= -\frac12\op{tr}(XY)$.
Here, $B(\cdot,\cdot)$ denotes the Killing form on $\sll(2,\C)$. 
Furthermore, $[X_1,X_2]=2X_3$, $[X_3,X_1]=2X_2$ and $[X_2,X_3]=2X_1$.

For our purposes, it will be sufficient to consider $\SO(3)$ as $\SU(2)/\{\pm I_2\}$, where $I_2$ denotes the $(2\times2)$-identity matrix. 
We have that $\SU(2)$ is diffeomorphic to the $3$-dimensional sphere $S^3$, and $\SO(3)$ is diffeomorphic to the $3$-dimensional real projective space $P^3(\R)$.

Let $G$ be either $\SU(2)$ or $\SO(3)$, thus $\mathfrak g= \su(2)$. 
For $a,b,c$ positive real numbers, set $g_{\abc}=g_A$ where $A=\diag\abc$, with respect to the basis $\{X_1,X_2,X_3\}$.
Thus, $g_{\abc}$ is the left-invariant metric on $G$ corresponding to the inner product on $\su(2)$ such that $\{aX_1,bX_2,cX_3\}$ is orthonormal.  
Equivalently, the matrix of $g_{\abc}$ with respect to the basis $\{X_1,X_2,X_3\}$ is given by \eqref{eq2:g_matrix}.
Consequently, the associated Casimir element associated to $g_{\abc}$ is given by
\begin{equation}\label{eq2:Casimir}
C_{\abc}:= C_A= a^2X_1^2+b^2 X_2^2+c^2 X_3^2 \in U(\sll(2,\C)). 
\end{equation}

It turns out that it is sufficient to consider only the metrics $g_{\abc}$ (see for instance \cite[\S4]{Milnor76}).

\begin{proposition}\label{prop2:left-invSU(2)}
Let $G$ be either $\SU(2)$ or $\SO(3)$.
Any left-invariant metric on $G$ is isometric to $g_{\abc}$ for some positive real numbers $a$, $b$ and $c$. 
Moreover, any permutation on the parameters $\abc$ does not change the isometry class; thus one can assume $a\geq b\geq c$. 
\end{proposition}

For a general compact Lie group $G$ and a left-invariant metric $g_A$ introduced in the previous section, it is well known that the volume of $(G,g_A)$ depends only on $|\det(A)|$. 
Consequently, for $G=\SU(2)$ or either $\SO(3)$ and for $a,b,c,a',b',c'$ positive real numbers, we have that
\begin{equation}\label{eq2:vol=abc}
\op{vol}(G,g_{\abc})=\op{vol}(G,g_{(a',b',c')})
\quad\text{if and only if}\quad
abc=a'b'c'.
\end{equation}

The next result was proved by Milnor~\cite[Thm.~4.3]{Milnor76}. 
Notice the change of notation with \cite[Thm.~4.3]{Milnor76}: $e_1=aX_1$, $e_2=bX_2$, $e_3=cX_3$, thus $\lambda_1=2bc/a$, $\lambda_2=2ac/b$ and $\lambda_3=2ab/c$.

\begin{proposition}\label{prop2:Ricci+scal}
Let $G$ be either $\SU(2)$ or $\SO(3)$ and let $a,b,c$ be positive real numbers. 
The scalar curvature at any point of $(G,g_{\abc})$ is given by 
\begin{align}\label{eq2:scalarcurvature}
\op{Scal}(G,g_{\abc}) 
  &:= 4(a^2+b^2+c^2)-2\left(\frac{b^2c^2}{a^2} +\frac{a^2c^2}{b^2} +\frac{a^2b^2}{c^2}\right). 
\end{align}
\end{proposition}

\section{The smallest Laplace eigenvalue}\label{sec:spectra}

In this section, we prove the expressions for the smallest non-zero eigenvalue of the Laplace--Beltrami operator on $\SU(2)$ and $\SO(3)$ endowed with a left-invariant metric. 
Although Proposition~\ref{prop2:spec_A} gives a theoretical expression for every positive eigenvalue, it is not clear which one is the smallest one. 
Lemma~\ref{lem3:lowerboundCasimireigenvalues} in Subsection~\ref{subsec:eigenvalues} gives some relations among them by using the Gershgorin Circle Theorem. 
Subsection~\ref{subsec:expressions} includes the proofs of Theorems~\ref{thm1:SUlambda_1} and \ref{thm1:SOlambda_1} and also Theorem~\ref{thm3:nexteigenvalues} which gives some information about the next small eigenvalues. 
The section ends with a full description of the spectrum of any $3$-dimensional Berger sphere.

\subsection{Eigenvalues of Casimir elements} \label{subsec:eigenvalues}
It follows from Proposition~\ref{prop2:spec_A} that the spectrum $\spec(\SU(2),g_{\abc})$ associated to the left-invariant metric $g_{\abc}$ strongly depends on the eigenvalues of $\pi(-C_{\abc})$ for every irreducible representation $\pi$ of $\SU(2)$ (see Section~\ref{sec:preliminaries} for notation).
We will obtain some lower bounds and relations among them.
Their proofs are simple but quite technical. 
The reader is encouraged to skip this subsection in his/her first reading.

It is well known that, for each non-negative integer $k$, there is up to equivalence exactly one irreducible representation of $\SU(2)$ of dimension $k+1$. 
We denote such a representation by $\pi_k$.
Furthermore, $\pi_k$ descends to a representation of $\SO(3)=\SU(2)/\{\pm I_2\}$ if and only if $k$ is even.
Consequently, a representative set for the unitary dual of $\SO(3)$ is given by $\{\pi_{2k}:k\geq0 \}$. 
In particular, every irreducible representation of $\SO(3)$ has odd dimension.

The representation $\pi_k$ of $\SU(2)$ is realized as the space $V_{\pi_k}$ of complex homogeneous polynomials of degree $k$ in two variables,  with the action given by (see for instance \cite[\S IV.1]{Knapp-book-beyond})
$$
(\pi_{k}(g)\cdot P) (\begin{smallmatrix} z \\ w \end{smallmatrix}) = P(g^{-1} (\begin{smallmatrix} z \\ w \end{smallmatrix})).
$$
Set $P_l(\begin{smallmatrix} z \\ w \end{smallmatrix}) = z^lw^{k-l}$ for every $0 \leq l\leq k$ and let us denote by $\mathcal B$ the ordered basis $\{P_0,\dots,P_k\}$ of $V_{\pi_k}$.
For $v\in U(\sll(2,\C))$, we will write $[\pi_k(v)]$ for the matrix corresponding to the linear transformation $\pi_k(v):V_{\pi_k}\to V_{\pi_k}$ with respect to $\mathcal B$. 

\begin{lemma}\label{lem3:pi_k(C)}
The entries of $\pi_k(-C_{\abc})$ with respect to the basis $\mathcal B$ are given by 
\begin{align*}
[\pi_k(-C_{\abc})]_{j,j} 
  &= (k-2(j-1))^2 a^2 + ((2j-1)k-2(j-1)^2)(b^2+c^2),\\
[\pi_k(-C_{\abc})]_{j-2,j} 
  &= (j-2)(j-1)(b^2-c^2),\\
[\pi_k(-C_{\abc})]_{j+2,j} 
  &= (k-j)(k+1-j)(b^2-c^2),
\end{align*}
for any $1\leq j\leq k+1$, and zero otherwise.
\end{lemma}

\begin{proof}
For $X =(X_{i,j})\in \su(2)$, we have that 
\begin{align*}
(\pi_k(X)\cdot P)(\begin{smallmatrix} z \\ w \end{smallmatrix}) 
&= \left.\frac{d}{dt}\right|_{t=0} (\pi_k(e^{tX})\cdot P)(\begin{smallmatrix} z \\ w \end{smallmatrix}) 
= \left.\frac{d}{dt}\right|_{t=0} P\big(e^{-tX}( \begin{smallmatrix} z \\ w \end{smallmatrix})\big) \\
&= -\frac{\partial P}{\partial z}(\begin{smallmatrix} z \\ w \end{smallmatrix})(zX_{1,1}+wX_{1,2})-\frac{\partial P}{\partial w}(\begin{smallmatrix} z \\ w \end{smallmatrix})(zX_{2,1}+wX_{2,2}).
\end{align*}
Furthermore, $e^{tX_1} = \left(\begin{smallmatrix} e^{it}& 0\\ 0& e^{-it} \end{smallmatrix}\right)$, $e^{tX_2} = \left(\begin{smallmatrix} \cos(t)&\sin(t) \\ -\sin(t)&\cos(t) \end{smallmatrix}\right)$, and $e^{tX_3} = \left(\begin{smallmatrix} \cos(t)&i\sin(t) \\ i\sin(t)&\cos(t) \end{smallmatrix}\right)$.
Taking in account the above remarks, one can check that $\pi_k(X_1) \cdot P_l = (k-2l)i \, P_l$, $\pi_k(X_2) \cdot P_l = -l\, P_{l-1} + (k-l)\, P_{l+1}$ and $\pi_k(X_3) \cdot P_l = -li\, P_{l-1} - (k-l)i\, P_{l+1}$, thus 
\begin{align*}
\pi_k(X_1)^2 \cdot P_l &= -(k-2l)^2 \, P_l,\\
\pi_k(X_2)^2 \cdot P_l &= l(l-1)\, P_{l-2} -\big((2l+1)k-2l^2\big)\, P_l + (k-l)(k-l-1)\, P_{l+2},\\
\pi_k(X_3)^2 \cdot P_l &= -l(l-1)\, P_{l-2} -\big((2l+1)k-2l^2\big)\, P_l - (k-l)(k-l-1)\, P_{l+2},
\end{align*}
for every $0\leq l\leq k$.
Here, it is understood that $P_l=0$ if $l<0$ or $l>k$. 
The proof follows by $\pi_k(-C_{\abc})\cdot P_l = -a^2 \,(\pi_k(X_1)^2\cdot P_l) -b^2 \,(\pi_k(X_2)^2\cdot P_l) -c^2 \,(\pi_k(X_3)^2\cdot P_l)$.
\end{proof}

% <   >

\begin{example}\label{ex3:eigenvalues}(\textsl{The Laplacian on the isotypical component of $\pi_0$, $\pi_1$ and $\pi_2$})
We now study in detail the first three cases. 
We have that 
\begin{align*}
[\pi_0(-C_{\abc})]&=(0),\\
[\pi_1(-C_{\abc})] 
  &= 
\left(\begin{array}{cc}
a^{2} + b^{2} + c^{2} & 0 \\
0 & a^{2} + b^{2} + c^{2}
\end{array}\right)
,\\
[\pi_2(-C_{\abc})]&= 
\left(\begin{array}{ccc}
4  a^{2} + 2  b^{2} + 2  c^{2} & 0 & 2  b^{2} - 2  c^{2} \\
0 & 4  b^{2} + 4  c^{2} & 0 \\
2  b^{2} - 2  c^{2} & 0 & 4  a^{2} + 2  b^{2} + 2  c^{2}
\end{array}\right)
.
\end{align*}
Write $A=\diag\abc$. 
One can easily check that the eigenvalues of the above three matrices are $\lambda_1^{\pi_0,A}=0$, $\lambda_1^{\pi_1,A} =\lambda_2^{\pi_1,A}= a^{2} + b^{2} + c^{2}$, and $\lambda_3^{\pi_2,A} := 4(a^{2} +  b^{2})$, $\lambda_2^{\pi_2,A} := 4(a^{2} +  c^{2})$, $\lambda_1^{\pi_2,A} := 4(b^{2} +  c^{2})$.
We note that $\lambda_1^{\pi_2,A}\leq \lambda_2^{\pi_2,A}\leq \lambda_3^{\pi_2,A}$ if $a\geq b\geq c>0$. 

According to Proposition~\ref{prop2:spec_A}, we have that
\begin{itemize}
\item $\pi_1(-C_{\abc})$ contributes the eigenvalue $a^{2} + b^{2} + c^{2}$ to $\spec(\SU(2),g_{\abc})$ with multiplicity four;
\item $\pi_2(-C_{\abc})$ contributes the eigenvalues $\lambda_1^{\pi_2,A}$, $\lambda_2^{\pi_2,A}$ and $\lambda_3^{\pi_2,A}$ to $\spec(\SU(2),g_{\abc})$ and to $\spec(\SO(3),g_{\abc})$, each of them with multiplicity three. 
\end{itemize}
Notices these eigenvalues could have greater multiplicity in the corresponding spectra.
Indeed, this is the case when the set $\mathcal E_\lambda(G,A)$ in Proposition~\ref{prop2:spec_A} has more than one element. 
\end{example}

\begin{remark}
When $k$ is odd, the representation $\pi_k$ of $\SU(2)$ is symplectic. 
This means in particular that the eigenvalues of $\pi_k(-C_{\abc})$ come in pairs.
Hence, for each odd integer $k$, $\pi_k(-C_{\abc})$ contributes  $(k+1)/2$ eigenvalues to $\spec(\SU(2),g_{\abc})$, each of them with multiplicity $2(k+1)$. 
\end{remark}

When $a,b,c$ are different, it does not seem feasible to have an explicit expression  for every eigenvalue of $\pi_k(-C_{\abc})$ for all $k$. 
The next lemma controls them by giving a lower bound. 

\begin{lemma}\label{lem3:lowerboundCasimireigenvalues}
Let $k$ be a positive integer. 
If $a\geq b\geq c>0$, then any eigenvalue $\lambda$ of $\pi_{k}(-C_{\abc})$ satisfies $\lambda\geq 2kb^2+k^2c^2$.
Moreover, if $k$ is odd, then $\lambda\geq a^2+(2k-1)b^2+k^2c^2$. 
\end{lemma}

\begin{proof}
The main tool will be classical and elementary result called \emph{Gershgorin Circle Theorem} (see for instance \href{https://en.wikipedia.org/wiki/Gershgorin_circle_theorem}{Wikipedia's page \emph{Gershgorin Circle Theorem}}).
It says that any eigenvalue of an $m\times m$ complex matrix $M$ lies in the union over $1\leq j\leq m$ of the disks in the complex plane with center $M_{j,j}$ and radius $\sum_{i\neq j} |M_{i,j}|$. 
Hence, any eigenvalue $\lambda$ of $\pi_k(-C_{\abc})$ lies in the disk of radius $[\pi_k(-C_{\abc})]_{j-2,j}+[\pi_k(-C_{\abc})]_{j+2,j}$ centered at $[\pi_k(-C_{\abc})]_{j,j}$, for some $1\leq j\leq k+1$.
Thus, for some $1\leq j\leq k+1$,
\begin{equation*}
\lambda \geq \extremo (k,j,\abc):= [\pi_k(-C_{\abc})]_{j,j} -[\pi_k(-C_{\abc})]_{j-2,j}-[\pi_k(-C_{\abc})]_{j+2,j}.
\end{equation*} 
Consequently, it remains to show that $\extremo (k,j,\abc)\geq 2kb^2+k^2c^2$ for every $k$ and also $\extremo (k,j,\abc)\geq a^2+(2k-1)b^2+k^2c^2$ for every $k$ odd, in both cases for all $1\leq j\leq k+1$.

We note that the expressions in Lemma~\ref{lem3:pi_k(C)} for the entries of $[\pi_k(-C_{\abc})]$ imply that $[\pi_k(-C_{\abc})]_{j-2,j}=0$ for $j=0,1$ and $[\pi_k(-C_{\abc})]_{j+2,j}=0$ for $j=k,k+1$, though these entries do not exist. 
This fact provides the necessary consistency to the previous paragraph.

For $1\leq j\leq k+1$, Lemma~\ref{lem3:pi_k(C)} implies that 
\begin{align}\label{eq:alpha}
\extremo (k,j,\abc)
&= (k-2(j-1))^2 a^2 + ((2j-1)k-2(j-1)^2)(b^2+c^2) \\
&\quad - (j-2)(j-1)(b^2-c^2)- (k-j)(k+1-j)(b^2-c^2) \notag\\
&= (k+2-2j)^2a^2+k^2 c^2 - ((k+2-2j)^2-2k)b^2 \notag \\
&= (k+2-2j)^2(a^2-b^2)+2kb^2+k^2 c^2 . \notag
\end{align}
Since $a\geq b\geq c>0$, the real function $j\mapsto \extremo (k,j,\abc)$ is quadratic with its minimum at $j=k/2+1$. 
Consequently, $\extremo (k,j,\abc) \geq \extremo (k,k/2+1,\abc) = 2kb^2+k^2 c^2$.
Moreover, when $k$ is odd, the minimum of $\extremo (k,j,\abc)$ for $j$ restricted to $\Z$ is clearly attained when $j=(k+1)/2$, thus $\extremo (k,j,\abc)\geq  a^2+(2k-1)b^2+k^2 c^2\geq a^2+b^2+k^2 c^2$ for every $1\leq j\leq k+1$. 
\end{proof}

\begin{remark}\label{rem3:tridiagonal}
We claim that the matrix $[\pi_k(C_{\abc})]$ is similar to two blocks of tridiagonal matrices. 
Let $P=(p_{i,j})$ and $Q=(q_{i,j})$ be the square matrices of size $\lfloor\frac{k+2}2 \rfloor$ and $\lfloor\frac{k+1}2 \rfloor$ respectively, given by $p_{i,j}=[\pi_k(C_{\abc})]_{2i-1,2j-1}$ and  $q_{i,j}=[\pi_k(C_{\abc})]_{2i,2j}$. 
Thus, $P$ (resp.\ $Q$) is the matrix resulting from $[\pi_k(C_{\abc})]$ after subtracting the even (resp.\ odd) rows and columns.
In particular, it turns out that $P$ and $Q$ are tridiagonal matrices (i.e.\ every entry $(i,j)$ with $|i-j|\geq 2$ vanishs). 
In other words, the matrix of $\pi_k(C_{\abc})$ with respect to the ordered basis $\{P_0, P_2, P_4,\dots\} \cup \{P_1,P_3, P_5,\dots\}$ is given by $\diag(P,Q)$, two blocks of tridiagonal matrices as asserted. 
This fact was already observed in \cite[Thm.~4.1]{Schueth17}.
\end{remark}

\subsection{Explicit expressions}\label{subsec:expressions}
We now prove Theorem~\ref{thm1:SUlambda_1} as a consequence of Proposition~\ref{prop2:spec_A} and Lemma~\ref{lem3:lowerboundCasimireigenvalues}.

\begin{proof}[Proof of Theorem~\ref{thm1:SUlambda_1}]
We have to prove that the smallest non-zero eigenvalue $\lambda_1(\SU(2),g_{\abc})$ in $\spec(\SU(2),g_{\abc})$ is equal to $\lambda_{\min}\abc := \min\{a^{2} + b^{2} + c^{2}, \; 4(b^{2} +  c^{2})\}$.
From Proposition~\ref{prop2:spec_A}, the set of eigenvalues of the Laplace--Beltrami operator on $(\SU(2),g_{\abc})$ is the union over $k$ of the set of eigenvalues of $\pi_{k}(-C_{\abc})$. 
Example~\ref{ex3:eigenvalues} shows that $a^{2} + b^{2} + c^{2}$ and $4(b^{2} +  c^{2})$ lie in $\spec(\SU(2),g_{\abc})$, thus  
$
\lambda_1(\SU(2),g_{\abc}) \leq \lambda_{\text{min}}\abc. 
$ 
Moreover, for every $k\geq2$, Lemma~\ref{lem3:lowerboundCasimireigenvalues} implies that any eigenvalue $\lambda$ of $\pi_k(C_{\abc})$ satisfies
\begin{equation}\label{eq3:lambdalowerbound}
\lambda\geq 2kb^2+k^2c^2\geq 4(b^2+c^2)\geq \lambda_{\text{min}}\abc.
\end{equation}
In particular, for $k\geq3$, we see that $\pi_{k}(-C_{\abc})$ does not contribute to the fundamental tone of $g_{(a,b,c)}$.
Hence $\lambda_1(\SU(2),g_{\abc}) =\lambda_{\text{min}}\abc$ since the only eigenvalue of $\pi_1({-C_{\abc}})$ is $a^{2} + b^{2} + c^{2}$.

We now compute the multiplicity of $\lambda_1(\SU(2),g_{\abc})$ in $\spec(\SU(2),g_{\abc})$.
From Proposition~\ref{prop2:spec_A}, this number is equal to 
\begin{equation}
\sum_{(k,j)\in\mathcal E} (k+1),
\end{equation}
where $\mathcal E=\{(k,j): k\geq1,\, 1\leq j\leq k+1,\, \lambda_1(\SU(2),g_{\abc})= \lambda_{j}^{\pi_k,\diag\abc}\}$. 
If $a^{2} + b^{2} + c^{2}<4(b^2+c^2)$, then $\lambda_1(\SU(2),g_{\abc})=a^{2} + b^{2} + c^{2}<\lambda_{j}^{\pi_k,\diag\abc}$ for all $k\geq2$ by \eqref{eq3:lambdalowerbound}, hence $\mathcal E=\{(1,1),(1,2)\}$ and therefore the multiplicity is equal to $2+2=4$. 

We now suppose  $4(b^2+c^2)<a^{2} + b^{2} + c^{2} = \lambda_j^{\pi_1,\diag\abc}$ for $j=1,2$, so $\lambda_1(\SU(2),g_{\abc})=4(b^2+c^2)$.
It follows that $a^2>3(b^2+c^2)$, thus $\lambda_1(\SU(2),g_{\abc})=\lambda_1^{\pi_2,\diag\abc} =4(b^2+c^2)< \lambda_2^{\pi_2,\diag\abc}=4(a^2+c^2)\leq \lambda_3^{\pi_2,\diag\abc}=4(a^2+b^2)$.
Furthermore,  for any $k\geq3$, \eqref{eq3:lambdalowerbound} gives $\lambda_1(\SU(2),g_{\abc}) = 4(b^2+c^2)< 6b^2+9c^2\leq  \lambda_{j}^{\pi_k,\diag\abc}$. 
Consequently, $\mathcal E=\{(2,1)\}$ and the multiplicity of $\lambda_1(\SU(2),g_{\abc})$ in $\spec(\SU(2),g_{\abc})$ is equal to $3$. 

We conclude the proof by assuming $a^{2} + b^{2} + c^{2}=4(b^2+c^2)$, so $\lambda_1(\SU(2),g_{\abc})=a^{2} + b^{2} + c^{2}=4(b^2+c^2)$ and $(1,1),(1,2),(2,1)\in\mathcal E$.
We claim that $\mathcal E=\{(1,1),(1,2),(2,1)\}$, which implies that the multiplicity of $\lambda_1(\SU(2),g_{\abc})$ in the spectrum of $\Delta_{g_{\abc}}$ is equal to $2+2+3=7$ as asserted. 
We have already seen in the previous case that $\lambda_1(\SU(2),g_{\abc}) = 4(b^2+c^2)<  \lambda_{j}^{\pi_k,\diag\abc}$ for all $k\geq3$.
Furthermore, $4(b^2+c^2)< 4(a^2+c^2)\leq 4(a^2+b^2)$ still holds if $a^{2} =3(b^2+c^2)$, and the proof is complete. 
\end{proof}

We next give more information than in Theorem~\ref{thm1:SUlambda_1} for low eigenvalues in $\spec(\SU(2),g_{\abc})$. 
To do that, we use the following different way to describe $\spec(M,g)$ for an arbitrary compact Riemannian manifold $(M,g)$.

\begin{notation}\label{not3:spec=set}
The spectrum $\spec(M,g)$ of the Laplace--Beltrami operator $\Delta_g$ on $M$ is the set (not a multiset anymore) of pairs $(\mu_j(M,g),m_j(M,g))$ for $j\in \N_0$, where 
$$
0<\mu_1(M,g)<\mu_2(M,g)<\mu_3(M,g)<\cdots
$$
are the eigenvalues of $\Delta_g$, and $m_j(M,g)$ is a positive integer indicating the multiplicity of $\mu_j(M,g)$ in $\spec(M,g)$. 
In the notation of \eqref{eq1:spectrum(M,g)}, $\mu_1(M,g)=\lambda_j(M,g)$ for all $1\leq j\leq m_1(M,g)$, $\mu_2(M,g)=\lambda_j(M,g)$ for all $m_1(M,g)+1\leq j\leq m_1(M,g)+m_2(M,g)$, and so on. 
\end{notation}

\begin{example}(\textsl{The fundamental tone of a homogeneous three-sphere})
In the language of Notation~\ref{not3:spec=set}, Theorem~\ref{thm1:SUlambda_1} is equivalent to 
\begin{equation}
(\mu_1(\SU(2),g_{\abc}), m_1(\SU(2),g_{\abc})) = 
\begin{cases}
(a^2+b^2+c^2,4)&\quad\text{if $a^2<3(b^2+c^2)$},\\ 
(a^2+b^2+c^2,7)&\quad\text{if $a^2=3(b^2+c^2)$},\\
(4(b^2+c^2),3)&\quad\text{if $a^2>3(b^2+c^2)$}.
\end{cases}
\end{equation}
\end{example}

It is also possible to give expressions for $(\mu_j(\SU(2),g_{\abc}), m_j(\SU(2),g_{\abc}))$ for higher $j$, though the division of cases becomes involved soon.
As an example, the next result shows some properties for the arrangement of the eigenvalues in $\spec(\SU(2),g_{\abc})$. 

% <   >

\begin{theorem}\label{thm3:nexteigenvalues}
Assume $a\geq b\geq c>0$.
The following assertions hold. 
\begin{enumerate}
 \item If $j$ satisfies $\mu_j(\SU(2),g_{\abc})=4(b^2+c^2)$, then $j=1$ or $j=2$. That is, $4(b^2+c^2)$ is the first or second positive eigenvalue of $\Delta_{g_{\abc}}$ on $\SU(2)$ without counting multiplicities. 
 
 \item Fix $b\geq c>0$. For $a\geq b$, let $j(a)$ be the positive integer satisfying that $a^2+b^2+c^2=\mu_{j(a)}(\SU(2),g_{\abc})$. Then $j(a)\to\infty$ when $a\to\infty$. 
\end{enumerate}
\end{theorem}

\begin{proof}
We first show (i). 
If $a^2\geq 3(b^2+c^2)$, then $\mu_1(\SU(2),g_{\abc})=4(b^2+c^2)$ by Theorem~\ref{thm1:SUlambda_1}, and the assertion follows in this case. 
We now assume $a^2 < 3(b^2+c^2)$. 
We want to prove that $\mu_2(\SU(2),g_{\abc})=4(b^2+c^2)$. 
We know that $\pi_1(-C_{\abc})$ has exactly one eigenvalue, namely, $a^2+b^2+c^2=\mu_1(\SU(2),g_{\abc})$ by Theorem~\ref{thm1:SUlambda_1}.
Hence, it remains to show that $4(b^2+c^2) \leq \lambda_{i}^{\pi_k,\diag{\abc}}$ for all $k\geq2$.
This immediately follows by Lemma~\ref{lem3:lowerboundCasimireigenvalues}. 

We now show (ii). 
Like in the proof of Lemma~\ref{lem3:lowerboundCasimireigenvalues}, Gershgorin Circle Theorem ensures that every eigenvalue $\lambda$ of $\pi_k(-C_{\abc})$ lies in a disk of center $[\pi_k(-C_{\abc})]_{j,j}$ and ratio $[\pi_k(-C_{\abc})]_{j-2,j}+[\pi_k(-C_{\abc})]_{j+2,j}$ for some $1\leq j\leq k+1$.
That is, 
\begin{equation}\label{eq3:intervals}
\lambda\in \bigcup_{j=1}^{k+1} \big(\alpha(k,j,\abc),\beta(k,j,\abc) \big),
\end{equation}
where $\beta(k,j,\abc) := [\pi_k(-C_{\abc})]_{j,j} +[\pi_k(-C_{\abc})]_{j-2,j}+[\pi_k(-C_{\abc})]_{j+2,j}.$
Moreover, the Gershgorin Circle Theorem also yields that if one of the intervals in \eqref{eq3:intervals} is disjoint to the other ones, then there is exactly one eigenvalue of $\pi_k(-C_{\abc})$ liying in this interval. 

Fix $k$ even. 
If $a$ is sufficiently large, then
$$
\beta(k,\tfrac k2+1,\abc) = b^2k^2+2c^2k <  (k+2-2j)^2(a^2-b^2)+2kb^2+k^2 c^2= \alpha(k,j,\abc)
$$ 
for all $j\neq k/2-1$, by Lemma~\ref{lem3:pi_k(C)}. 
Consequently, the $(j=k/2-1)$-th interval in \eqref{eq3:intervals} 
$$
\big(\alpha(k,\tfrac k2+1,\abc),\beta(k,\tfrac k2+1,\abc)\big)=\big( 2kb^2+k^2c^2,k^2b^2+2kc^2\big), 
$$ 
is disjoint to the other intervals for $a$ big enough. 
Hence, $\pi_k(-C_{\abc})$ has exactly one eigenvalue, say $\lambda_1^{\pi_k,\diag\abc}$, less than $b^2k^2+c^2k \leq a^2+b^2+c^2$ for $a$ large enough.  
We conclude that 
$
\lim_{a\to\infty} j(a)\geq \# \{\lambda_1^{\pi_k,\diag\abc}:k\geq0 \text{ even}\}=\infty,
$
as asserted.
\end{proof}

\begin{proof}[Proof of Theorem~\ref{thm1:SOlambda_1}]
The unitary dual of $\SO(3)$ is given by $\{\pi_{2k} :k\geq0\}$.
By Proposition~\ref{prop2:spec_A}, the set of eigenvalues of $(\SO(3),g_{\abc})$ is the union over $k$ of the set of eigenvalues of $\pi_{2k}(-C_{\abc})$. 
Example~\ref{ex3:eigenvalues} ensures that $4(b^{2} +  c^{2})$ is in $\spec(\SO(3),g_{\abc})$, thus $\lambda_1(\SO(3),g_{\abc})\leq 4(b^{2} +  c^{2})$.
Lemma~\ref{lem3:lowerboundCasimireigenvalues} implies for all $k\geq2$ and any $j$ that
\begin{equation}\label{eq3:SOlambdalowerbound}
\lambda_{j}^{\pi_{2k},\diag\abc} \geq 4kb^2+4k^2c^2\geq 8b^2+16 c^2>4(b^2+c^2).
\end{equation}

It only remains to compare $4(b^2+c^2)$ with the rest of the eigenvalues of $\pi_2(-C_\abc)$, namely, $4(a^2+c^2)$ and $4(a^2+b^2)$.
The assumption $a\geq b\geq c >0$ clearly forces
\begin{equation}\label{eq3:eigenvaluespi_2}
4(b^2+c^2)\leq 4(a^2+c^2)\leq 4(a^2+b^2),
\end{equation}
which proves $\lambda_1(\SO(3),g_{\abc})= 4(b^{2} +  c^{2})$.
Moreover, \eqref{eq3:SOlambdalowerbound} and \eqref{eq3:eigenvaluespi_2} immediately imply the assertion on the multiplicity of $\lambda_1(\SO(3),g_{\abc})$ in $\spec(\SO(3),g_{\abc})$. 
\end{proof}

% <   >

Any Berger sphere is isometric to $(\SU(2),g_{(a,b,b)})$ for some positive numbers $a,b$. 
We note that by Lemma~\ref{lem3:pi_k(C)} the condition $b=c$ implies the matrix of $\pi_{k}(-C_{\abc})$ is diagonal with diagonal entries
\begin{equation}
\nu_{k,j}(a,b):=a^2(k-2j)^2  + 2b^2((2j+1)k-2j^2)
\end{equation}
for $0\leq j\leq k$.  
Hence, Proposition~\ref{prop2:spec_A} and Lemma~\ref{lem3:pi_k(C)} immediately imply the following complete description of the spectra of three-dimensional Berger spheres. 
%It is not clear to the author whether this result has been already published in somewhere else. 
This result can be found in \cite[Lem.~4.1]{Tanno79}.

\begin{proposition}\label{prop3:specBerger}
For every $a,b$ positive real numbers, we have that 
\begin{align*}
\op{Spec}(\SU(2),g_{(a,b,b)}) &=
\bigcup_{k\geq0} \big\{\!\big\{
\underbrace{\nu_{k,j}(a,b),\dots,\nu_{k,j}(a,b)}_{(k+1)\text{-times}}:0\leq j\leq k
\big\}\!\big\},\\
\op{Spec}(\SO(3),g_{(a,b,b)}) &=
\bigcup_{k\geq0\atop k\text{ even}} \big\{\!\big\{
\underbrace{\nu_{k,j}(a,b),\dots,\nu_{k,j}(a,b)}_{(k+1)\text{-times}}:0\leq j\leq k
\big\}\!\big\}.
\end{align*}
\end{proposition}

The spectrum of the Dirac operator on any Berger sphere was explicitly computed by B\"ar~\cite{Bar92}.

\section{Estimates on homogeneous 3-spheres} \label{sec:estimates}
Let $G$ be either $\SU(2)$ or $\SO(3)$ and let $g$ be a left-invariant metric on $G$.
This section deals with the estimates for $\lambda_1(G,g)$ in terms of $\op{diam}(G,g)^{-2}$ announced in Subsection~\ref{subsec1:estimates}. 
The proofs of these results are in Subsection~\ref{subsec:estimations}. 
They use the expressions for $\lambda_1(G,g)$ given in Theorems~\ref{thm1:SUlambda_1} and \ref{thm1:SOlambda_1} as well as an estimation of $\op{diam}(G,g)$ given in  Subsection~\ref{subsec:diameter}.

% <   >

\subsection{Diameter}\label{subsec:diameter}
An expression for the diameter of an arbitrary left-invariant metric $g_{\abc}$ on $\SU(2)$ or $\SO(3)$ is not avaiable at the moment.
However, we will bound it by expressions in terms of the parameters $a$, $b$ and $c$.
The proof is based on the proof of \cite[Prop.~7.1]{EldredgeGordinaSaloff17}, which shows that $\op{diam}(\SU(2),g_{\abc})$ is comparable with $b^{-1}$ for all $a\geq b\geq c> 0$. 
That is, there are positive real numbers $D_0$ and $D_\infty$ such that 
\begin{equation}
D_0 b^{-1} \leq \op{diam}(\SU(2),g_{\abc}) \leq D_\infty b^{-1}.
\end{equation}
We will refine the above bounds for $\op{diam}(G,g_{\abc})$ (see Proposition~\ref{prop4:diam-acotado-Berger}), obtaining as a consequence explicit and optimal values for $D_0$ and $D_\infty$ (see Corollary~\ref{cor4:diameter*b}). 

The next result gives an explicit expression for the diameter of any Berger sphere, and also for the analogous case when $G=\SO(3)$.
The expression \eqref{eq4:diamBergerSO} below is proved by Podobryaev and Sachkovin~\cite[Thm.~4]{PodobryaevSachkov16}, while the expression \eqref{eq4:diamBergerSU} is shown by Podobryaev~ \cite[Thm.~1]{Podobryaev17}. 
In their notation, the multisets $\{\!\{I_1,I_2,I_3\}\!\}$ and $\{\!\{1/(2a)^2,1/(2b)^2,1/(2c)^2\}\!\}$ coincide.

\begin{proposition}\label{prop4:diameterBerger}
We have that
\begin{align}\label{eq4:diamBergerSU}
\op{diam}(\SU(2),g_{\abc}) &=
\begin{cases}
\dfrac{\pi}{b} &\quad\text{if } a=b\geq c,\\[3mm]
\dfrac{\pi}{a} &\quad\text{if } a>b=c\geq a/\sqrt{2},\\[3mm]
\dfrac{\pi}{2b\sqrt{1-b^2/a^2}} &\quad\text{if }  a/\sqrt{2}>b=c,
\end{cases} \\
\label{eq4:diamBergerSO}
\op{diam}(\SO(3),g_{\abc}) &=
\begin{cases}
\dfrac{\pi}{2b} &\quad\text{if } a>b=c,\\[3mm]
\dfrac{\pi}{2c} &\quad\text{if } a=b\geq c\geq b/\sqrt{2},\\[3mm]
\dfrac{\pi}{b}\sqrt{1+\frac{1}{4(c^2/b^2-1)}} &\quad\text{if }  a=b>\sqrt{2}c.
\end{cases}
\end{align}
\end{proposition}

The following well-known result is included for completeness in the case of interest to us. 

\begin{lemma}\label{lem4:diamg1g2}
Let $G$ be a compact connected Lie group and let $g_1$ and $g_2$ be left-invariant metrics on $G$.
If $g_1(X,X)\leq g_2(X,X)$ for all left-invariant vector field $X$ on $G$, then $\op{diam}(G,g_1)\leq \op{diam}(G,g_2)$.
\end{lemma}

\begin{proof}
Let $p,q\in G$ satisfying that $\op{diam}(G,g_1) = \op{dist}_{g_1}(p,q)$, and let $\gamma:[0,1]\to G$ be a geodesic with respect to $g_2$ realizing the distance between $p$ and $q$, that is, $\gamma(0)=p$, $\gamma(1)=q$ and $\op{dist}_{g_2}(p,q) = \op{lenght}_{g_2}(\gamma):= \int_0^1 g_2(\dot\gamma(t),\dot\gamma(t))^{1/2} dt$. 
Hence, 
\begin{align*}
\op{diam}(G,g_2) 
  &\geq \int_0^1 g_2(\dot\gamma(t),\dot\gamma(t))^{1/2} dt 
  \geq \int_0^1 g_1(\dot\gamma(t),\dot\gamma(t))^{1/2} dt 
  \geq \op{dist}_{g_1}(p,q)=\op{diam}(G,g_1),
\end{align*}
and the proof is complete. 
\end{proof}

\begin{proposition}\label{prop4:diam-acotado-Berger}
For all $a\geq b\geq c>0$, we have that
\begin{align}\label{eq4:diam-acotadoSU}
\frac{\pi}{b}\geq 
\op{diam}(\SU(2),g_{\abc}) \geq &
\begin{cases}
\dfrac{\pi}a & \text{if }a\leq \sqrt{2} b, \\
\dfrac{\pi}{2b\sqrt{1-b^2/a^2}} & \text{if }a>\sqrt{2} b .
\end{cases} \\
\label{eq4:diam-acotadoSO}
\frac{\pi}{2b} \leq  \op{diam}(\SO(3),g_{\abc}) \leq&  
\left\{
\begin{array}{ll}
\dfrac{\pi}{2c} \quad& \text{if }c\leq \sqrt{2} b,\\[3mm]
\dfrac{\pi}{b} \sqrt{\frac{1}{4(c^2/b^2-1)}} \quad& \text{if } c>\sqrt{2} b.
\end{array}
\right.
\end{align}
\end{proposition}

\begin{proof}
Let $G$ be either $\SU(2)$ or $\SO(3)$.
By \eqref{eq2:g_matrix}, one has that
\begin{equation*}
g_{(a,b,b)}(X,X)\leq g_{\abc}(X,X)\leq g_{(b,b,c)} (X,X)
\end{equation*}
for all left-invariant vector fields $X$ on $G$.
Lemma~\ref{lem4:diamg1g2} now gives
\begin{equation*}
\op{diam}(G,g_{(a,b,b)})\leq \op{diam}(G,g_{\abc}) \leq \op{diam}(G,g_{(b,b,c)}).
\end{equation*}
Proposition~\ref{prop4:diameterBerger} now completes the proof. 
\end{proof}

\begin{corollary}\label{cor4:diameter*b}
For all $a\geq b\geq c>0$, we have that
\begin{align*}
\frac{\pi}{2b}<&\op{diam}(\SU(2), g_{\abc}) \leq \frac{\pi}{b}
&\text{ and }&& 
\frac{\pi }{2b} \leq &\op{diam}(\SO(3), g_{\abc}) < \frac{\sqrt{3}\pi }{2b}.
\end{align*}
Moreover, the strict (resp.\ non-strict) inequalities are asymptotically sharp (resp.\ sharp). 
\end{corollary}

\begin{proof}
The non-strict inequalities were already proved in Proposition~\ref{prop4:diam-acotado-Berger}. 
The strict inequality for $\SU(2)$ (resp.\ $\SO(3)$) follows by taking the infimum (resp.\ supremum) of the function at the left in \eqref{eq4:diam-acotadoSU} (resp.\ at the right in \eqref{eq4:diam-acotadoSO}). 
Since this function is decreasing (resp.\ increasing), its infimum (resp.\ supremum) is given by 
\begin{equation*}
\lim_{a\to\infty} \dfrac{\pi}{2b\sqrt{1-b^2/a^2}}=\frac{\pi}{2b}
\qquad \left( \text{ resp.\ } 
\lim_{c\to0^+} \dfrac{\pi}{b}\sqrt{1+\frac{1}{4(c^2/b^2-1)}} = \frac{\sqrt{3}\pi}{2b}
\right).
\end{equation*}
Moreover, these inequalities are asymptotically sharp taking $\abc=(a,b,b)$ with $a\to\infty$ and $\abc=(b,b,c)$ with $c\to0^+$ respectively. 
\end{proof}

% <   >

\subsection{Estimations}\label{subsec:estimations}
The expression for $\lambda_1(\SU(2),g_{\abc})$ in Theorem~\ref{thm1:SUlambda_1} implies the following result, which refines \cite[Cor.~8.6]{EldredgeGordinaSaloff17} by giving optimal values to the bounds.
The case $\SO(3)$ is also considered. 

\begin{corollary}\label{cor4:lambda1/b^2}
Under the assumption $a\geq b\geq c>0$, we have that 
\begin{align}\label{eq4:SUlambda_1bounded}
2b^2 < \lambda_1(\SU(2),g_{\abc})\leq 8 b^2, \\
\label{eq4:SOlambda_1bounded}
4b^2 < \lambda_1(\SO(3),g_{\abc})\leq 8 b^2. 
\end{align}
Moreover, the inequalities at the right (resp.\ left) are sharp (resp.\ asymptotically sharp). 
\end{corollary}

\begin{proof}
Since $\lambda_1(\SU(2),g_{\abc})=\min\{a^2+b^2+c^2,4(b^2+c^2)\}$ by Theorem~\ref{thm1:SUlambda_1}, it follows that $\lambda_1(\SU(2),g_{\abc})\leq 4(b^2+c^2)\leq 8b^2$ and $\lambda_1(\SU(2),g_{\abc})\geq 2b^2+c^2>2b^2$. 

The equality at the right of \eqref{eq4:SUlambda_1bounded} is attained when $b=c$ and $a^2\geq6b^2$.
Furthermore, by considering $\abc$ such that $a=b\geq c$, we obtain that $\lambda_1\abc=2b^2+c^2$, which approaches to $2b^2$ when $c$ goes to $0$. 
The proof of \eqref{eq4:SOlambda_1bounded} follows in the same way as for \eqref{eq4:SUlambda_1bounded}. 
\end{proof}

It follows immediately from the previous result and Corollary~\ref{cor4:diameter*b} that 
\begin{align}
\pi^2/2 &<\lambda_1(\SU(2),g) \, \op{diam}(\SU(2),g)^2 \leq 8\pi^2,\\
\pi^2 &<\lambda_1(\SO(3),g) \,\op{diam}(\SO(3),g)^2 < 6\pi^2,
\end{align}
for any left-invariant metric $g$ on the corresponding Lie group. 
Although the inequalities at the right (resp.\ left) above were obtained by two sharp (resp.\ asymptotically sharp) inequalities, we will improve them (Theorem~\ref{thm1:conjSU2SO3}) by using better estimates for the diameter. 
We first consider Berger spheres, where we have explicit expressions for the lowest eigenvalue and the diameter.

\begin{corollary}\label{cor4:BergerSU}
For every Berger sphere $(\SU(2),g)$, we have that 
\begin{equation*}
\frac{(1+\sqrt{3}/2)\pi^2}{\op{diam}(\SU(2),g)^2} \leq \lambda_1(\SU(2),g)   \leq \frac{3\pi^2}{\op{diam}(\SU(2),g)^2}. 
\end{equation*}
Moreover, the inequality at the right (resp.\ left) is attained at the round sphere $a=b=c$ (resp.\ when $b=c$ and $a^2=({\sqrt{3}-1}) b^2/2$).
\end{corollary}

\begin{proof}
To facilitate the reading we write the explicit expressions from Theorem~\ref{thm1:SUlambda_1} and Proposition~\ref{prop4:diameterBerger} for a Berger sphere as follows: 
\begin{align}\label{eq4:lambda_1-Berger}
\lambda_1(\SU(2),g_{\abc}) &=
\begin{cases}
2b^2+c^2 \quad&\text{if } a=b\geq c,\\
a^2+2b^2 \quad&\text{if } b=c,\, b^2\leq a^2\leq 6b^2,\\
8b^2 \quad&\text{if } b=c,\, a^2\geq 6b^2,
\end{cases}
\\
\label{eq4:diam-Berger}
\op{diam}(\SU(2),g_{\abc})^2 &=
\begin{cases}
\pi^2/b^2 \quad&\text{if } a=b\geq c,\\
\pi^2/a^2 \quad&\text{if } b=c,\; b^2\leq a^2\leq 2b^2,\\
\dfrac{\pi^2}{4b^2(1-b^2/a^2)} \quad&\text{if } b=c,\; a^2 \geq 2b^2.
\end{cases}
\end{align}

We claim that the image of the function 
$$
\abc\mapsto \lambda_1(\SU(2),g_{\abc}) \op{diam}(\SU(2),g_{\abc})^2
$$ 
over the corresponding subsets of the domain are the intervals
\begin{equation}\label{eq4:claim}
\begin{cases}
(2\pi^2, 3\pi^2] \quad& \text{if }a=b\geq c,\\
[2\pi^2, 3\pi^2] \quad& \text{if }b=c,\, b^2\leq a^2\leq 2b^2,\\
[\frac{1+\sqrt{3}}2\pi^2, \frac{12}{5} \pi^2] \quad& \text{if }b=c,\, 2b^2\leq a^2\leq 6b^2,\\
(2\pi^2, \frac{12}{5} \pi^2] \quad& \text{if }b=c,\, 6b^2\leq a^2.
\end{cases}
\end{equation}
This would imply the first assertion, and its proof also the second one. 

If $a=b\geq c$, then $\lambda_1(\SU(2),g_{\abc}) \op{diam}(\SU(2),g_{\abc})^2 =\frac{2b^2+c^2}{b^2} \,  \pi^2$ by \eqref{eq4:lambda_1-Berger} and \eqref{eq4:diam-Berger}.
Moreover, its maximum is $3\pi^2$ attained when $c=b$ and its infimum is $2\pi^2$ asymptotically attained when $c$ goes to $0$. 
The second and fourth row in \eqref{eq4:claim} follow in a similarly straightforward way as the first one  by \eqref{eq4:lambda_1-Berger} and \eqref{eq4:diam-Berger}. 

Suppose that $b=c$ and $2b^2\leq a^2\leq 6b^2$. 
Then, $\lambda_1(\SU(2),g_{\abc}) \op{diam}(\SU(2),g_{\abc})^2$  equals
\begin{equation}
\frac{(a^2+2b^2)\pi^2}{4b^2(1-b^2/a^2)} = \frac{(1+2b^2/a^2)\pi^2}{4b^2/a^2(1-b^2/a^2)}
= \frac{(1+2x)\pi^2}{4x(1-x)} =:f(x),
\end{equation}
where $x=b^2/a^2$. 
By assumption, $1/6\leq x\leq 1/2$. 
A simple study of this function shows that 
$$
(1+\sqrt{3}/2)\pi^2 =f(\tfrac{\sqrt{3}-1}2)\leq f(x)\leq f(1/6)=12\pi^2/5
$$ 
for all $1/6\leq x\leq 1/2$. 
This completes the proof of \eqref{eq4:claim} and the corollary. 
\end{proof}

% <   >

We are now in a position to prove the main result of this section.

\begin{proof}[Proof of Theorem~\ref{thm1:conjSU2SO3}]
The upper bound in \eqref{eq1:conjSU(2)} as well as the lower bound in \eqref{eq1:conjSO(3)} follow immediately by Corollaries~\ref{cor4:diameter*b} and \ref{cor4:lambda1/b^2}.
Throughout the proof, Theorems~\ref{thm1:SUlambda_1} and \ref{thm1:SOlambda_1} as well as Proposition~\ref{prop4:diam-acotado-Berger} will be used repeatedly times without any comment, to give expressions for $\lambda_1(\SU(2),g_{\abc})$ or bounds to $\op{diam}(G,g_{\abc})^2$. 

In order to establish the lower bound in \eqref{eq1:conjSU(2)} (i.e.\ $G=\SU(2)$) we divide the proof depending the position of $a^2$ with respect to $2b^2$ and $3(b^2+c^2)$. 

If $a^2\leq 2b^2$, then $\lambda_1(\SU(2),g_{\abc})=a^2+b^2+c^2$ and 
$
\op{diam}(\SU(2),g_{\abc})^2 \geq  \pi^2/a^2,
$
thus $\lambda_1(\SU(2),g_{\abc})\op{diam}(\SU(2),g_{\abc})^2$ is greater than or equal to 
$$
\frac{a^2+b^2+c^2}{a^2}\, \pi^2 >  (1+ {b^2}/{a^2})\, \pi^2 \geq \frac{3\pi^2}2.
$$

Suppose now that $2b^2\leq a^2\leq 3(b^2+c^2)$. 
We have that $\lambda_1(\SU(2),g_{\abc})=a^2+b^2+c^2$ and 
$
\op{diam}(\SU(2),g_{\abc})^2 \geq \frac{\pi^2}{4b^2(1-b^2/a^2)}.
$
We deduce that $\lambda_1(\SU(2),g_{\abc})\op{diam}(\SU(2),g_{\abc})^2$ is greater than or equal to 
$$
\frac{(a^2+b^2+c^2)\pi^2}{4b^2(1-b^2/a^2)} \geq
\frac{\pi^2}{3}\, \frac{a^2}{b^2(1-b^2/a^2)}\geq 
\frac{\pi^2}{3}\, \frac{1}{b^2/a^2(1-b^2/a^2)}\geq \frac{4\pi^2}{3},
$$
since the map $x\mapsto ({x(1-x)})^{-1}$ restricted to $1/6\leq x\leq 1/2$ attains its minimum at $x=1/2$. 

If $a^2\geq 3(b^2+c^2)$,  $\lambda_1(\SU(2),g_{\abc})=4(b^2+c^2)$ and $
\op{diam}(\SU(2),g_{\abc})^2 \geq \frac{\pi^2}{4b^2(1-b^2/a^2)}
$, 
thus $\lambda_1(\SU(2),g_{\abc})\op{diam}(\SU(2),g_{\abc})^2$ is greater than or equal to 
$$
\frac{4(b^2+c^2)\pi^2}{4b^2(1-b^2/a^2)} >
\frac{\pi^2}{(1-b^2/a^2)}
> \pi^2.
$$

Combining the last three inequalities, we conclude that 
$$
\lambda_1(\SU(2),g_{\abc})> \frac{\pi^2}{\op{diam}(\SU(2),g_{\abc})^2}
$$
as asserted. This completes the proof of \eqref{eq1:conjSU(2)}.

It only remains to establish the upper bound in \eqref{eq1:conjSO(3)} (i.e.\ $G=\SO(3)$). 
In this case, we always have $\lambda_1(\SO(3),g_{\abc})=4(b^2+c^2)$, thus it will be convenient to divide the proof according to the position of $c^2$ with respect to $b^2/2$. 
If $2c^2\geq b^2$, then  
$
\op{diam}(\SO(3),g_{\abc})^2 \leq  \frac{\pi^2}{4c^2},
$
thus $\lambda_1(\SU(2),g_{\abc})\op{diam}(\SU(2),g_{\abc})^2$ is smaller than or equal to 
$$
\frac{4(b^2+c^2)\pi^2}{4c^2}=(b^2/c^2+1)\pi^2  \leq 3\pi^2.
$$

We now assume $2c^2\leq b^2$. 
This gives 
\begin{equation*}
\lambda_1(\SO(3),g_{\abc})
\op{diam}(\SO(3),g_{\abc})^2 
\leq  %4(b^2+c^2)\,\frac{\pi^2}{b^2} \left(1+\frac{1}{4(c^2/b^2-1)}\right)\\ = 
\frac{\pi^2(1+c^2/b^2)(4c^2/b^2-3)}{c^2/b^2-1}
\leq (9-4\sqrt{2})\pi^2.
\end{equation*}
The last inequality follows by a standard study of the maximum of the function $x\mapsto \frac{(1+x)(4x-3)}{x-1}$ restricted to $0< x\leq 1/2$. 
One can check that such a maximum is attained at $x=1-1/\sqrt{2}$. 

Combining the last two inequalities, we conclude that 
$$
\lambda_1(\SO(3),g_{\abc})\leq  \frac{(9-4\sqrt{2})\pi^2}{\op{diam}(\SO(3),g_{\abc})^2}
$$
as asserted, and the proof of Theorem~\ref{thm1:conjSU2SO3} is complete. 
\end{proof}

We conclude this section by considering left-invariant metrics on products of $\SU(2)$ and $\SO(3)$ given by the product of left-invariant metrics in each of the components.

\begin{proposition} \label{prop4:estimate-product}
Let $m$ and $n$ be non-negative integers.
Let 
$$
G=
\underbrace{\SU(2)\times\dots\times \SU(2)}_{m\text{-times}} \times \underbrace{\SO(3) \times\dots\times \SO(3)}_{n\text{-times}}
$$
endowed with the left-invariant metric $g$ given by 
$$
g_{(a_1,b_1,c_1)} \times \dots \times g_{(a_{m},b_{m},c_{m})}  
\times 
g_{(a_1',b_1',c_1')} \times \dots \times g_{(a_{n}',b_{n}',c_{n}')},
$$
where $a_i\geq b_i\geq c_i>0$ for all $1\leq i \leq m$ and $a_j'\geq b_j'\geq c_j'>0$ for all $1\leq j \leq n$.
Then
\begin{equation}\label{eq4:estimateproduct}
\frac{\pi^2}{\op{diam}(G,g)^2} <\lambda_1(G,g)  \leq  \frac{(8m+6n)\pi^2}{\op{diam}(G,g)^2}.
\end{equation}
\end{proposition}

\begin{proof}
It is well known that every eigenvalue of the Laplace--Beltrami operator on a product of closed compact Riemannian manifolds $(M,g):=(M_1,g_1)\times \dots\times (M_n,g_n)$ is of the form $\lambda^{(1)}+\dots+\lambda^{(n)}$, for $\lambda^{(j)}$ in $\spec(M_j,g_j)$. 
For any $j$, since $M_j$ is compact, constant functions on $M_j$ are eigenfunctions of $\Delta_{g_j}$ with eigenvalue $0$.
Hence $0\in\spec(M_j,g_j)$ for all $j$, thus the smallest non-zero eigenvalue of $\Delta_g$ is given by 
\begin{equation}\label{eq4:lambda_1producto}
\min\left\{\lambda_1(M_1,g_1),\dots,\lambda_1(M_n,g_n) \right\}.
\end{equation}
The explicit expressions for $\lambda_1(\SU(2),g_{(a_i,b_i,c_i)})$  and $\lambda_1(\SO(3),g_{(a_j',b_j',c_j')})$ gives \begin{equation}\label{eq4:lambda_1producto2}
\lambda_1(G,g) = \min\left\{ 
\{a_i^2+b_i^2+c_i^2,4(b_i^2+c_i^2):1\leq i\leq m\} 
\cup
\{4(b_j'^2+c_j'^2):1\leq j\leq n\}
\right\}.
\end{equation}

For arbitrary closed Riemannian manifolds $(M_i,g_i)$, it is a simple matter to check that 
\begin{equation}\label{eq4:diameter-producto}
\op{diam}\big((M_1,g_1)\times \dots\times (M_n,g_n)\big)^2 = \sum_{i=1}^n \op{diam}(M_i,g_i)^2>\op{diam}(M_j,g_j)^2
\end{equation}
for all $j$. 
Consequently, \eqref{eq4:lambda_1producto} clearly forces that $\lambda_1(G,g)\op{diam}(G,g)^2$ is greater than the minimum among the numbers 
\begin{align*}
\lambda_1(\SU(2),g_{(a_i,b_i,c_i)})\op{diam}(\SU(2),g_{(a_i,b_i,c_i)})^2
&\quad\text{for }1\leq i\leq m, \text{ and}\\
\lambda_1(\SO(3),g_{(a_j',b_j',c_j')})\op{diam}(\SO(3),g_{(a_j',b_j',c_j')})^2
&\quad\text{for }1\leq j\leq n.                                                                                             \end{align*}
Since each of them is greater than $\pi^2$ by Theorem~\ref{thm1:conjSU2SO3}, we obtain the lower bound in \eqref{eq4:estimateproduct}.

Set $B=\min\{b_1,\dots,b_{m},b_{1}',\dots,b_{n}'\}$. 
Similarly as in Corollary~\ref{cor4:lambda1/b^2}, 
\eqref{eq4:lambda_1producto2} gives $\lambda_1(G,g) \leq 8 B^2$. 
Furthermore, \eqref{eq4:diameter-producto} and Corollary~\ref{cor4:diameter*b} show that 
\begin{equation}\label{eq4:diam/B}
\op{diam}(G,g)^2 
\leq  \sum_{i=1}^{m} \frac{\pi^2}{b_i^2} +\sum_{j=1}^{n} \frac{3\pi^2}{4b_j'^2} 
\leq  \frac{(m+3n/4)\pi^2}{B^2},
\end{equation}
which completes the proof. 
\end{proof}

\begin{example} \label{ex4:no-unif-upper-bound} (\textsl{There is no uniform bound on $C_G$ in Conjecture~\ref{conj:EGS}.})
For each positive integer $n$, set $G_n=\SU(2)^n$.
Let $g_n$ be the left-invariant metric on $G_n$ given by 
$
g_{(1,1,1)} \times \dots \times g_{(1,1,1)}.
$
In other words, $(G_n,g_n)$ is the product of $n$-copies of the round $3$-sphere of radius $1$. 
Hence, $\lambda_1(G_n,g_n)= 3$ by \eqref{eq4:lambda_1producto2} and $\op{diam}(G_n,g_n)^2=n\pi^2$ by \eqref{eq4:diameter-producto}. 
This clearly forces that $C_{G_n}\geq 3n\pi^2$ which goes to infinity when $n$ does it. 
\end{example}

% <   >

\section{Global spectral rigidity}\label{sec:uniquiness}
In this section we prove Theorem~\ref{thm1:unicidad}, which ensures that two isospectral left-invariant metrics on $\SU(2)$ or $\SO(3)$ are necessarily isometric. 
This result has been already shown by Schmidt and Sutton~\cite{SchmidtSutton13} by using heat invariants.
They proved that the first four heat invariants distinguish the isometry class of every left-invariant metric on $\SU(2)$.

To prove Theorem~\ref{thm1:unicidad} we will show in  Theorem~\ref{thm5:unicidad} that any left-invariant metric on $\SU(2)$ or $\SO(3)$ is distinguished by the volume, the scalar curvature and the lowest non-zero eigenvalue with its multiplicity, which are spectral invariant by the next well-known result.

\begin{proposition}\label{prop5:specinvariants}
Let $G$ be either $\SU(2)$ or $\SO(3)$.
If $\spec(G,g_{\abc}) = \spec(G,g_{(a',b',c')})$, then 
$
\op{vol}(G,g_{\abc})=\op{vol}(G,g_{(a',b',c')})$ and $\op{Scal}(G,g_{\abc})=\op{Scal}(G,g_{(a',b',c')}).
$
\end{proposition}

\begin{proof}
It is well known that the spectrum of an arbitrary compact Riemannian manifold determines its volume and its total scalar curvature. 
This information is obtained by the first and second heat invariants respectively.
Furthermore, when the manifold is homogeneous (e.g.,\ a compact Lie group endowed with a left-invariant metric), the scalar curvature is constant.
Hence, the scalar curvature is equal to the total scalar curvature divided by the volume. 
Since both terms are spectral invariants, so is the scalar curvature. 
\end{proof}

%$<   >$

\begin{theorem}\label{thm5:unicidad}
Let $G$ be either $\SU(2)$ or $\SO(3)$.
The volume, the scalar curvature, and the smallest non-zero eigenvalue with its multiplicity mutually distinguish isometry classes of left-invariant metrics on $G$. 
\end{theorem}

\begin{proof}
We fix a left-invariant metric $g_{\abc}$ on $\SU(2)$ with $a\geq b \geq c>0$. 
Let $g_{(x,y,z)}$ be a left-invariant metric on $\SU(2)$ such that $x\geq y\geq z>0$ and it has the same volume, scalar curvature and smallest non-zero eigenvalue with its multiplicity as $g_{\abc}$. 
We want to show that $(x,y,z)=\abc$. 
By \eqref{eq2:vol=abc} and Proposition~\ref{prop2:Ricci+scal}, we have that 
\begin{align}
abc&=xyz, \label{eq5:vol}\\
2(a^2+b^2+c^2)-\frac{(ab)^4+(ac)^4+(bc)^4}{(abc)^2} &=
2(x^2+y^2+z^2)-\frac{(xy)^4+(xz)^4+(yz)^4}{(xyz)^2}.\label{eq:scal}
\end{align}

We first consider the case $a^2\leq  3(b^2+c^2)$.
We have $a^2+b^2+c^2\leq 4(b^2+c^2)$, thus Theorem~\ref{thm1:SUlambda_1} yields that $\lambda_1(\SU(2),g_{\abc})=a^2+b^2+c^2$ with multiplicity four or seven according $a^2+b^2+c^2< 4(b^2+c^2)$ or $a^2+b^2+c^2= 4(b^2+c^2)$ respectively. 
By assumption, $\lambda_1(\SU(2),g_{\xyz})=a^2+b^2+c^2$ with multiplicity four or seven in $\spec(\SU(2),g_{\xyz})$, thus
\begin{equation}\label{eq5:lambda1caso1}
a^2+b^2+c^2 = x^2+y^2+z^2
\end{equation}
and $x^2\leq 3(y^2+z^2)$ by Theorem~\ref{thm1:SUlambda_1}, since otherwise, $x^2> 3(y^2+z^2)$ gives multiplicity three to the first non-zero eigenvalue.
To prove this case, we need to show that \eqref{eq5:vol}, \eqref{eq:scal}, \eqref{eq5:lambda1caso1}, $a\geq b \geq c>0$ and $x\geq y \geq z>0$ imply $\xyz=\abc$.

Substituting \eqref{eq5:vol} and \eqref{eq5:lambda1caso1} into \eqref{eq:scal}, $(ab)^4+(ac)^4+(bc)^4 = (xy)^4+(xz)^4+(yz)^4$. 
Thus
\begin{align*}
((ab)^2+(ac)^2+(bc)^2)^2 
  &= (ab)^4+(ac)^4+(bc)^4 + 2(abc)^2(a^2+b^2+c^2) \\
  &= (xy)^4+(xz)^4+(yz)^4 + 2(xyz)^2(x^2+y^2+z^2) \\
  &= ((xy)^2+(xz)^2+(yz)^2)^2.
\end{align*}
Since both numbers are positive, we obtain that 
\begin{align*}
(ab)^2+(ac)^2+(bc)^2
  &= (xy)^2+(xz)^2+(yz)^2 = x^2(y^2+z^2)+(yz)^2 \\
  &= x^2(a^2+b^2+c^2-x^2)+ \left(\tfrac{abc}{x}\right)^2.
\end{align*}
Hence,
\begin{align*}
0 &=x^4(a^2+b^2+c^2-x^2)+ (abc)^2 - x^2((ab)^2+(ac)^2+(bc)^2) \\
&= -x^6 +x^4(a^2+b^2+c^2)-x^2(a^2b^2+a^2c^2+b^2c^2)+a^2b^2c^2 \\
&= -(x^2-a^2)(x^2-b^2)(x^2-c^2).
\end{align*}
Since all parameters are positive, $x$ must be $a$, $b$ or $c$. 

If $x=a$, then $b^2+c^2=y^2+z^2$ and $bc=xy$ by \eqref{eq5:vol} and \eqref{eq5:lambda1caso1}, thus $(b+c)^2=(y+z)^2$ and $(b-c)^2=(y-z)^2$, thus $b+c=y+z$ and $b-c=y-z$, therefore $y=b$ and $z=c$ as required. 
When $x=b$ or $x=c$ is assumed, the same procedure as above will show that $\xyz$ is a permutation of $\abc$, but then the assumptions $a\geq b \geq c>0$ and $x\geq y \geq z>0$ imply a that $\xyz=\abc$ or a contradiction in case $a\neq b$ or $a\neq c$ respectively. 
This concludes the proof of the case $a^2< 3(b^2+c^2)$.

We now assume $a^2> 3(b^2+c^2)$, thus $4(b^2+c^2)<a^2+b^2+c^2$ and Theorem~\ref{thm1:SUlambda_1} shows that $\lambda_1(\SU(2),g_{\abc}) = 4(b^2+c^2)$ with multiplicity three in $\spec(\SU(2),g_{\abc})$. 
Similarly as in the previous case, Theorem~\ref{thm1:SUlambda_1} ensures that 
\begin{align}\label{eq5:lambda1caso2}
4(b^2+c^2) &= 4(y^2+z^2)\\
x^2 &> 3(y^2+z^2). \label{eq:condicioncaso2}
\end{align}

Combining \eqref{eq5:vol} and \eqref{eq5:lambda1caso2} with \eqref{eq:scal} we deduce that 
\begin{equation*}
2a^2 -\frac{(ab)^4+(ac)^4+(bc)^4}{(abc)^2} =
2x^2 -\frac{(xy)^4+(xz)^4+(yz)^4}{(abc)^2}.
\end{equation*}
By standard manipulations, still using \eqref{eq5:vol} and \eqref{eq5:lambda1caso2}, we obtain that 
\begin{align*}
0&=2(abc)^2(x^2-a^2)+(ab)^4+(ac)^4+(bc)^4-x^4((y^2+z^2)^2-2y^2z^2) -(yz)^4\\
&=2(abc)^2(x^2-a^2)+(ab)^4+(ac)^4+(bc)^4-x^4(b^2+c^2)^2+2x^2(abc)^2-\frac{(abc)^4}{x^4}\\
&= -(x^2-a^2)\Big(x^6(b^2+c^2)^2+x^4a^2(b^2-c^2)^2- b^4c^4(x^2+a^2)\Big).
\end{align*}
Thus $x^2-a^2=0$ or $x^6(b^2+c^2)^2+x^4a^2(b^2-c^2)^2- b^4c^4(x^2+a^2)=0$. 
Since $x$ is positive, the first case gives $x=a$, which implies that $y=b$ and $z=c$ as above. 

It remains to prove that the fact of $x$ being a root of 
$$
g(t):=t^6(b^2+c^2)^2+t^4a^2(b^2-c^2)^2- b^4c^4(t^2+a^2)
$$ 
gives a contradiction
We proceed to make a simple study of the polynomial $g(t)$ as in a first course of calculus. 
It is a simple matter to check the following properties: $g(t)$ is an even function, $g(a)=a^2b^4(a^4-c^4)+a^2c^4(a^4-b^4)>0$, $g''(t)=0$ has two (opposite) solutions, $g$ has a local maximum at $t=0$ with $g(0)=-a^2b^4c^4<0$ and exactly one local minimum at some point in the interval $(0,a)$.
On account of the above information, we conclude that $g(t)$ has exactly one positive root, say $x$. 
We have that
\begin{align*}
	g(\sqrt[3]{abc}) 
	&=(abc)^2(b^2+c^2)^2+(abc)^{4/3}a^2(b^2-c^2)^2- b^4c^4((abc)^{2/3}+a^2) \\
	&\geq (abc)^2(b^2+c^2)^2- 2a^2b^4c^4 
	= (abc)^2 (b^4+c^4)>0. 
\end{align*}
Then $x<\sqrt[3]{abc}$, thus $xyz\leq x^3 <abc$, contrary to \eqref{eq5:vol}. 
This concludes the proof for $\SU(2)$. 

The proof for $\SO(3)$ is actually contained above. 
Indeed, Theorem~\ref{thm1:SOlambda_1} gives immediately \eqref{eq5:lambda1caso2}.
Hence, by proceeding as above, \eqref{eq5:vol}, \eqref{eq:scal} and \eqref{eq5:lambda1caso2} imply that $\abc=\xyz$
\end{proof}

%$<   >$

\section{Locally rigid homogeneous solutions of the Yamabe problem} \label{sec:Yamabe}

In this short section, we prove Theorem~\ref{thm1:Yamabe}. 
The following result is essential.

\begin{proposition}\label{prop6:SUScalcondition}
For every left-invariant metric $g$ on $\SU(2)$, we have that \begin{equation}\label{eq6:strictinequality}
\frac{\op{Scal}(\SU(2),g)}{2} \leq \lambda_1(\SU(2),g).
\end{equation}
Moreover, the equality holds if and only if $(\SU(2),g)$ is a round sphere. 
\end{proposition}

\begin{proof}
By Theorem~\ref{thm1:SUlambda_1} and Proposition~\ref{prop2:Ricci+scal}, \eqref{eq6:strictinequality} is equivalent to 
\begin{equation}\label{eq6:inequality1}
2(a^2+b^2+c^2)- \frac{a^4b^4+a^4c^4+b^4c^4}{a^2b^2c^2} \leq \min\{a^2+b^2+c^2, 4(b^2+c^2)\}
\end{equation}
for all $a\geq b\geq c>0$.
We first show that the left-hand side in \eqref{eq6:inequality1} is less than or equal to $a^2+b^2+c^2$ for all $a,b,c$. 
Indeed, by elementary manipulations, this is equivalent to $(a^2+b^2+c^2)a^2b^2c^2 \leq a^4b^4+a^4c^4+b^4c^4$. 
The last condition follows immediately by Muirhead's inequality (see for instance \href{https://en.wikipedia.org/wiki/Muirhead%27s_inequality}{Wikipedia's page \emph{Muirhead's inequality}}), since the vector $(4,4,0)$ majorizes $(4,2,2)$. 
Moreover, equality holds if and only if $a=b=c$, as asserted. 

It remains to show that the left-hand side in \eqref{eq6:inequality1} is less than $4(b^2+c^2)$, or equivalently
\begin{equation}
a^2 < \frac{a^4b^4+a^4c^4+b^4c^4}{2a^2b^2c^2} +b^2+c^2 
= a^2\frac{(b^4+c^4)}{2b^2c^2} + \frac{b^2c^2}{2a^2}+b^2+c^2,
\end{equation}
which follows immediately since $b^4+c^4\geq 2b^2c^2$. 
\end{proof}

\begin{proposition}\label{prop6:SOScalcondition}
For every left-invariant metric $g$ on $\SO(3)$, we have that \begin{equation}\label{eq6:SOstrictinequality}
\frac{\op{Scal}(\SO(3),g)}{2}< \lambda_1(\SO(3),g).
\end{equation}
\end{proposition}

\begin{proof}
The proof follows immediately from Proposition~\ref{prop6:SUScalcondition} and the fact that $\lambda_1(\SU(2),g_{\abc}) \leq \lambda_1(\SO(3),g_{\abc})$ 
for all $a,b,c$, with strict inequality when $a=b=c$. 
\end{proof}

\begin{proof}[Proof of Theorem~\ref{thm1:Yamabe}]
When a left-invariant metric $g$ on $\SU(2)$ has not constant sectional curvature, the strict inequality in \eqref{eq6:strictinequality} implies that $g$ is a locally rigid solution of the Yamabe problem according to \cite[Cor.~2.5]{BettiolPiccione13a}. 
This proves Theorem~\ref{thm1:Yamabe} for a non-round left-invariant metric on $\SU(2)$. 
The case when $g$ is any left-invariant metric on $\SO(3)$ follows from Proposition~\ref{prop6:SOScalcondition}. 
\end{proof}

%$<   >$

\bibliographystyle{plain}

\begin{thebibliography}{EGS18}

\bibitem[An05]{Anderson05}
	{\sc M. Anderson}.
	{\it On uniqueness and differentiability in the space of {Y}amabe metrics.}
	Commun. Contemp. Math. \textbf{7}:3 (2005), 299--310. 
	DOI: \href{http://dx.doi.org/10.1142/S0219199705001751} {10.1142/S0219199705001751}.

\bibitem[B\"a92]{Bar92}
	{\sc C. B\"ar}.
	{\it The Dirac operator on homogeneous spaces and its spectrum on 3-dimensional lens spaces.}
	Arch. Math. \textbf{59} (1992), 65--79.
	DOI: \href{http://dx.doi.org/10.1007/BF01199016} {10.1007/BF01199016}.

\bibitem[Be]{Berger-panoramic}
	{\sc M. Berger}.
	{A panoramic view of {R}iemannian geometry.}
	Springer-Verlag, Berlin, 2003.
	DOI: \href{http://dx.doi.org/10.1007/978-3-642-18245-7} {10.1007/978-3-642-18245-7}.		

\bibitem[BB82]{Berard-BergeryBourguignon82}
	{\sc L. B\'erard-Bergery, J.-P. Bourguignon}.
	{\it Laplacians and {R}iemannian submersions with totally geodesic fibres}.
	Illinois J. Math. \textbf{26}:2 (1982), 181--200.

\bibitem[BP13]{BettiolPiccione13a}
	{\sc R. Bettiol, P. Piccione}.
	{\it Bifurcation and local rigidity of homogeneous solutions to the {Y}amabe problem on spheres.}
	Calc. Var. Partial Differential Equations \textbf{47}:3--4 (2013), 789--807.
	DOI: \href{http://dx.doi.org/10.1007/s00526-012-0535-y} {10.1007/s00526-012-0535-y}.

\bibitem[BP13b]{BettiolPiccione13b}
	{\sc R. Bettiol, P. Piccione}.
	{\it Multiplicity of solutions to the {Y}amabe problem on collapsing Riemannian submersions.}
	Pacific J. Math. \textbf{266}:1 (2013), 1--21.
	DOI: \href{http://dx.doi.org/10.2140/pjm.2013.266.1} {10.2140/pjm.2013.266.1}.
	
\bibitem[EGS18]{EldredgeGordinaSaloff17}
	{\sc N. Eldredge, M. Gordina, L. Saloff-Coste}.
	{\it Left-invariant geometries on $\mathrm{SU}(2)$ are uniformly doubling}.
	Geom. Funct. Anal., \textbf{28}:5 (2018), 1321--1367.
	DOI: \href{http://dx.doi.org/10.1007/s00039-018-0457-8} {10.1007/s00039-018-0457-8}.

\bibitem[JL17]{JudgeLyons17}
	{\sc C. Judge, R. Lyons}.
	{\it Upper bounds for the spectral function on homogeneous spaces via volume growth}.
	\href{https://arxiv.org/abs/1704.01108v2} {arXiv:1704.01108v2} (2017).
	
\bibitem[Kn]{Knapp-book-beyond}
	{\sc A.W. Knapp}.
	{Lie groups beyond an introduction.}
	{\it Progr. Math.} \textbf{140}.
	Birkh\"auser Boston Inc., 2002.

\bibitem[La18]{Lauret-globalrigid}
	{\sc E.A. Lauret}.
	{\it Spectral uniqueness of bi-invariant metrics on symplectic groups}.
	Transform. Groups, in press since June 2018.
	DOI: \href{http://dx.doi.org/10.1007/s00031-018-9486-5} {10.1007/s00031-018-9486-5}.

\bibitem[LP87]{LeeParker87}
	{\sc J. Lee, T. Parker}.
	{\it The {Y}amabe problem.}
	Bull. Amer. Math. Soc. (N.S.) \textbf{17}:1 (1987), 37--91.
	DOI: \href{http://dx.doi.org/10.1090/S0273-0979-1987-15514-5} {10.1090/S0273-0979-1987-15514-5}.

\bibitem[Li80]{Li80}
	{\sc P. Li}.
	{\it Eigenvalue estimates on homogeneous manifolds.}
	Comment.\ Math.\ Helvetici \textbf{55} (1980), 347--363.
	DOI: \href{http://dx.doi.org/10.1007/BF02566692} {10.1007/BF02566692}.

\bibitem[LY80]{LiYau80}
 	{\sc P. Li, S.-T. Yau}.
 	{\it Estimates of eigenvalues of a compact {R}iemannian manifold}.
 	In Geometry of the {L}aplace operator ({P}roc. {S}ympos. {P}ure {M}ath., {U}niv. {H}awaii, {H}onolulu, {H}awaii, 1979), Proc. Sympos. Pure Math., XXXVI (1980), 205--239.

\bibitem[Li]{Lichnerowicz}
	{\sc A. Lichnerowicz}.
	{G\'eom\'etrie des groupes de transformations.}
	{\it Travaux et Recherches Math\'ematiques, III. Dunod, Paris}, 1958.

\bibitem[LL10]{LingLu10}
 	{\sc J. Ling, Z. Lu}.
 	{\it Bounds of eigenvalues on {R}iemannian manifolds}.
 	In Trends in partial differential equations, Adv. Lect. Math. (ALM) \textbf{10} (2010), 241--264. Int. Press, Somerville, MA.		
	
\bibitem[Mi76]{Milnor76}
	{\sc J. Milnor}.
	{\it Curvatures of left invariant metrics on lie groups.}
	Adv. Math. \textbf{21}:3 (1976), 293--329.
	DOI: \href{http://dx.doi.org/10.1016/S0001-8708(76)80002-3}{10.1016/S0001-8708(76)80002-3}.

\bibitem[NR03]{NikonorovRodionov03}
	{\sc Yu.G. Nikonorov, E.D. Rodionov}.
	{\it Compact homogeneous {E}instein 6-manifolds.}
	Differential Geom. Appl. \textbf{19}:3 (2003), 369--378.
	DOI: \href{http://dx.doi.org/10.1016/S0926-2245(03)00051-2}{10.1016/S0926-2245(03)00051-2}.
	
\bibitem[Ob62]{Obata62}
	{\sc M. Obata}.
	{\it Certain conditions for a {R}iemannian manifold to be isometric with a sphere.}
	J. Math. Soc. Japan \textbf{14} (1962), 333--340.
	DOI: \href{http://dx.doi.org/10.2969/jmsj/01430333}{10.2969/jmsj/01430333}.
	
\bibitem[Po18]{Podobryaev17}
	{\sc A.V. Podobryaev}.
	{\it Diameter of the Berger Sphere}.
	Math. Notes \textbf{103}:5--6 (2018), 846--851. 
	DOI: \href{http://dx.doi.org/10.1134/S0001434618050188} {10.1134/S0001434618050188}.
	
\bibitem[PS16]{PodobryaevSachkov16}
	{\sc A.V. Podobryaev, Yu.L. Sachkov}.
	{\it Cut locus of a left invariant {R}iemannian metric on {${\rm SO}_3$} in the axisymmetric case}.
	J. Geom. Phys. \textbf{110} (2016), 436--453.
	DOI: \href{http://dx.doi.org/10.1016/j.geomphys.2016.09.005}  {10.1016/j.geomphys.2016.09.005}.

\bibitem[Sa81]{Sakai81}
	{\sc T. Sakai}.
	{\it Cut loci of {B}erger's spheres}.
	Hokkaido Math. J. \textbf{10}:1 (1981), 143--155. 
	DOI: \href{http://dx.doi.org/10.14492/hokmj/1381758107}  {10.14492/hokmj/1381758107}.	

\bibitem[SS14]{SchmidtSutton13}
	{\sc B. Schmidt, C. Sutton}.
	{\it Detecting the moments of inertia of a molecule via its rotational spectrum, II}.
	Preprint available at \href{http://users.math.msu.edu/users/schmidt/}{Schmidt's web page.} (2014).

\bibitem[Sch17]{Schueth17}
	{\sc D. Schueth}.
	{\it Generic irreducibilty of Laplace eigenspaces on certain compact Lie groups.}
	Ann. Global Anal. Geom. \textbf{52}:2 (2017), 187--200.
	DOI: \href{http://dx.doi.org/10.1007/s10455-017-9553-5}{10.1007/s10455-017-9553-5}.

\bibitem[Ta79]{Tanno79}
	{\sc S. Tanno}.
	{\it The first eigenvalue of the {L}aplacian on spheres.}
	T\v ohoku Math. J. (2) \textbf{31}:2 (1979), 179--185. 
	DOI: \href{http://dx.doi.org/10.2748/tmj/1178229837} {10.2748/tmj/1178229837}.

\bibitem[Ur79]{Urakawa79}
	{\sc H. Urakawa}.
	{\it On the least positive eigenvalue of the {L}aplacian for compact group manifolds.}
	J. Math. Soc. Japan \textbf{31}:1 (1979), 209--226.
	DOI: \href{http://dx.doi.org/10.2969/jmsj/03110209} {10.2969/jmsj/03110209}.
	
\bibitem[Ya99]{Yang99}
	{\sc D. Yang}.
	{\it Lower bound estimates of the first eigenvalue for compact manifolds with positive {R}icci curvature.}
	Pacific J. Math. \textbf{190}:2 (1999), 383--398. 
	DOI: \href{http://dx.doi.org/10.2140/pjm.1999.190.383} {10.2140/pjm.1999.190.383}.

\bibitem[ZY84]{ZhongYang84}
	{\sc J.Q. Zhong, H.C. Yang}.
	{\it On the estimate of the first eigenvalue of a compact   {R}iemannian manifold.}
	Sci. Sinica Ser. A \textbf{27}:12 (1984), 1265--1273.
	
\end{thebibliography}

\end{document}